\documentclass[default]{sn-jnl}

\usepackage{graphicx}%
\usepackage{multirow}%
\usepackage{amsmath,amssymb,amsfonts}%
\usepackage{amsthm}%
\usepackage{mathrsfs}%
\usepackage[title]{appendix}%
\usepackage{xcolor}%
\usepackage{textcomp}%
\usepackage{manyfoot}%
\usepackage{booktabs}%
\usepackage{enumerate}%
\usepackage{algorithm}%
\usepackage{algorithmicx}%
\usepackage{algpseudocode}%
\usepackage{listings}%
\usepackage{caption}
\usepackage{subcaption}
\usepackage{graphicx}
\usepackage{comment}
\usepackage{hhline} 

\theoremstyle{thmstyleone}%
\newtheorem{theorem}{Theorem}
\newtheorem{proposition}[theorem]{Proposition}%

\theoremstyle{thmstyletwo}%
\newtheorem{remark}{Remark}%

\theoremstyle{thmstylethree}%

\newcommand{\HH}{\ensuremath{{\mathcal H}}}
\newcommand{\M}{\ensuremath{{\mathcal M}}}
\newcommand{\Y}{\ensuremath{{\mathcal Y}}}

\newcommand{\TT}{\ensuremath{\mathbb{T}}}

\newcommand{\funp}{\ensuremath{\ell}}
\newcommand{\emp}{\ensuremath{{\varnothing}}}

\newcommand{\Id}{\ensuremath{\operatorname{Id}}}

\newcommand{\RR}{\ensuremath{\mathbb{R}}}
\newcommand{\RP}{\ensuremath{\left[0,+\infty\right[}}
\newcommand{\RM}{\ensuremath{\left]-\infty,0\right]}}

\newcommand{\RPP}{\ensuremath{\left]0,+\infty\right[}}

\newcommand{\RX}{\ensuremath{\left]-\infty,+\infty\right]}}

\newcommand{\NN}{\ensuremath{\mathbb N}}

\newcommand{\weakly}{\ensuremath{\:\rightharpoonup\:}}

\newcommand{\ran}{\ensuremath{\text{\rm ran}\,}}

\newcommand{\pinf}{\ensuremath{{+\infty}}}

\newcommand{\dom}{\ensuremath{\text{\rm dom}\,}}

\newcommand{\prox}{\ensuremath{\text{\rm prox}}}

\newtheorem{problem}[theorem]{Problem}
\newcommand{\scal}[2]{{\left\langle{{#1}\mid{#2}}\right\rangle}}
\newcommand{\minimize}[2]{\ensuremath{\underset{\substack{{#1}}}%
		{\text{\rm minimize}}\;\;#2}}
\newcommand{\menge}[2]{\big\{{#1}~\big |~{#2}\big\}} 
 
\newtheorem{algo}[theorem]{Algorithm}

\begin{document}

\title[Forward-backward algorithm for functions with locally 
Lipschitz gradient: applications to mean field games. ]{Forward-backward algorithm for functions with locally  Lipschitz gradient: applications to mean field games.}


\author[1]{\fnm{Luis} \sur{M. Brice\~{n}o-Arias}}\email{luis.briceno@usm.cl}
\equalcont{These authors contributed equally to this work.}

\author*[2]{\fnm{Francisco} \sur{J. Silva}}\email{francisco.silva@unilim.fr}

\author[3]{\fnm{Xianjin} \sur{Yang}}\email{yxjmath@caltech.edu}
\equalcont{These authors contributed equally to this work.}

\affil[1]{\orgdiv{Departamento de Matem\'atica}, \orgname{Universidad T\'ecnica Federico Santa Mar\'ia}, \orgaddress{\city{Santiago}, \country{Chile}}}

\affil*[2]{\orgdiv{Laboratoire XLIM}, \orgname{Universit\'e de Limoges}, \orgaddress{\city{Limoges}, \postcode{87060}, \country{France}}}

\affil[3]{\orgdiv{Computing and Mathematical Sciences}, \orgname{California Institute of Technology}, \orgaddress{\state{California}, \country{USA}}}


\abstract{In this paper, we provide a generalization of the forward-backward
	splitting algorithm for minimizing the sum of a proper convex lower semicontinuous function and a differentiable convex function whose gradient satisfies a locally Lipschitz-type 
	condition. We prove the convergence of our method and derive a linear convergence rate when the differentiable function is locally strongly convex. We recover classical results in the case when the gradient of the differentiable function is globally Lipschitz continuous and an already known linear convergence rate when the function is globally strongly convex. We apply the algorithm to approximate equilibria of variational mean field game systems with local couplings. Compared with some benchmark algorithms to solve these problems, our numerical tests show similar performances in terms of the number of iterations but an important gain in the required computational time.}




\maketitle

\section{Introduction}

In this paper, we aim to solve the following problem. 
\vspace{0.2cm}
\begin{problem}
	\label{prob:main}
	Let $\HH$ be a real Hilbert space, let $\psi\colon\HH\to\RX$ be a proper lower semicontinuous convex function, and
	let $\varphi\colon\HH\to\RR$ be a convex G\^ateaux differentiable 
	function such that, for every 
	$x\in\HH$ and every $M\in\RPP$,
	there exist
	$\mu_{M}(x)\in\RP$ and 
	$L_{M}(x)\in 
	\left[\mu_{M}(x),\pinf\right[$ such that, for all $y, z\in B_M(x)$, 
	\begin{equation}
		\label{e:condgrad} 
		\mu_{M}(x) 
		\|y-z\|^2 \leq 
		\scal{y-z}{\nabla \varphi(y)- 
			\nabla \varphi(z)} \le  L_{M} 
		(x)
		\|y-z\|^2,
	\end{equation}
	where $B_M(x)$ is the
	open ball centered at $x$ with radius $M$.
	The problem is to
	\begin{equation}
		\label{e:mainprob}
		\minimize{x\in\HH}{\psi(x)+\varphi(x)},
	\end{equation}
	under the assumption that the set of 
	solutions to \eqref{e:mainprob}, denoted by $S\subset\HH$, is
	nonempty.
\end{problem}
\vspace{0.2cm}

This problem appears in several domains, including mean field games \cite{BAKS}, optimal transport problems \cite{benamoucarlier15,Papadakis_Peyre_Oudet_14}, image and signal processing \cite{Nelly,mms05}, control theory \cite{Bala63}, among others (see also \cite{mms05} and the references therein for more applications).
In the particular case when $\nabla \varphi$ is globally Lipschitz continuous, 
a standard algorithm for solving \eqref{e:mainprob} is the \textit{forward-backward splitting} (FBS), which finds its 
roots in the projected gradient method \cite{levitin1966constrained} (case 
$\psi=\iota_C$ for some nonempty closed convex set $C$). In the 
context of variational inequalities appearing in some PDEs, a 
generalization of the projected gradient method is proposed in 
\cite{BreSib, Mercier, Sibony}.
FBS combines a gradient step (forward) on $\varphi$ and a proximal (backward) step on $\psi$. More precisely, given $x_0\in\HH$, FBS iterates
\begin{equation}
	(\forall n\in\NN)\quad x_{n+1}=\prox_{\gamma\psi}\big(x_n-\gamma\nabla\varphi(x_n)\big),
\end{equation}
where $\gamma>0$ is known as the step-size and, for every $f\in\Gamma_0(\HH)$, $\prox_{f}\colon\HH\to\HH$ assigns to every $x\in\HH$ the unique solution to the lower semicontinuous strongly convex function $f+\|\cdot-x\|^2/2$. 
The weak convergence of the sequence 
generated by FBS to a solution to \eqref{e:mainprob} is guaranteed 
provided that $\gamma\in\left]0,2/L\right[$, where $L>0$ is the globally Lipschitz constant of $\nabla\varphi$ (see, e.g., \cite[Theorem~23.14]{19.Livre1}). One of the central arguments to prove the convergence is Baillon-Haddad's theorem \cite{MR500279}.
It asserts that globally Lipschitz continuous gradients of convex functions are cocoercive, from which it is proved that $\{x_n\}_{n\in\NN}$
is a F\'{e}jerian sequence, i.e., its distance to any solution is decreasing with $n\in\NN$.
If, in addition, $\varphi$ is strongly convex, FBS converges linearly to the unique solution to 
\eqref{e:mainprob} \cite{Nelly,Taylor}. However, the globally Lipschitz continuity on $\nabla\varphi$ is quite restrictive in applications. 

An approach to weaken the globally Lipschitz continuity of 
$\nabla\varphi$ is to use linesearch procedures to compute the step-size at each iteration of FBS (see, e.g., \cite{Yunier1,Yunier2,Salzo} and the references therein). In this setting, the convergence is then ensured under weaker conditions on $\nabla\varphi$ as, e.g., the uniform continuity on weakly compact sets \cite[Theorem~3.18]{Salzo}. However, each evaluation made in the linesearch procedure can be costly, e.g., when the proximity operator of $\psi$ is not simple to compute. Moreover, 
depending on the linesearch parameters, the resulting step-size can be very small, affecting the efficiency of the algorithm.

In this paper, we provide a new approach to guarantee the convergence of FBS when $\nabla\varphi$ satisfies \eqref{e:condgrad}. Our approach relies on a recent refinement of Baillon-Haddad's theorem for convex sets \cite{VPA},
which allows us to use the cocoercivity property of $\nabla\varphi$ in balls. This permits us to use similar arguments than in the globally Lipschitz continuous case to prove the convergence of FBS under the assumption \eqref{e:condgrad}. In addition, we can estimate an upper bound for the step-size and a linear convergence rate in the presence of strong convexity, as in the globally Lipschitz case. We also recover the classic convergence result for FBS when $\nabla\varphi$ is globally Lipschitz continuous, as well as the linear convergence rate in \cite{Nelly,Taylor} when, in addition, $\varphi$ is strongly convex.

Another contribution of this paper is the application of the proposed algorithm to approximate equilibria of  variational Mean Field Games (MFGs) with local couplings. The main purpose of MFGs theory, introduced independently by Lasry and Lions in~\cite{Lasry_Lions_2006i,MR2271747,Lasry_Lions_2007}, and by Caines, Huang, and Malham\'e in~\cite{Huang_et_all,HMC06}, is to describe the asymptotic behaviour of Nash equilibria of non-cooperative symmetric dynamic differential games with a large number of indistinguishable  players. We refer to~\cite{MR4214773,MR3752669,MR3753660,MR3559742,MR3195844}, and the references therein, for a general overview on MFGs theory including analytic and probabilistic aspects, as well as their numerical approximation and applications in crowd motion, economics, and finance. A particular class of MFGs, called {\it variational} MFGs, characterizes the aforementioned Nash equilibria in terms of the first order optimality condition of an associated optimization problem. This viewpoint opens the door to the application of variational techniques to establish the existence of MFG equilibria~\cite{MR3408214,cardaliaguet_meszaros_santambrogio_2018,MR3420414,meszaros_silva_2018} and to approximate them numerically by using convex optimization methods~\cite{MR2888257,Benamou_et_al_2017,MR4534446,MR3968547,BAKS,Lavigne_Pfeiffer_22,Nurbekyan_et_all_2020}.

In the framework of ergodic and variational MFGs with monotone local couplings, we consider the finite difference discretization introduced in~\cite{Gomesex}. 
Under assumptions ensuring the existence of a unique classical solution to the MFGs system, it is shown in \cite{Achdou_Camilli_Capuzzo_Dolcetta_13} that the solution to the finite difference scheme converges as the discretization step tends to zero. Let us also mention the contribution~\cite{Achdou_Porretta_16} dealing with the convergence of solutions to the scheme in the framework of weak solutions.  It turns out that this finite difference discretization preserves the variational structure, i.e. it corresponds to the optimality condition of a convex optimization problem. We apply the forward-backward algorithm to two dual formulations of this problem and we compare their performance in the case of a first order system involving a logarithmic coupling and which admits an explicit solution (see~\cite{Gomesex}). In the case of second order ergodic MFGs with power and logarithmic couplings, the forward-backward method is compared with state-of-the-art algorithms and we show, in this particular instance, a similar behavior in terms of the numbers of iterations but an important improvement in terms of computational time.

The remainder of the article is organized as follows. Section~\ref{sec:preliminaries} reviews the necessary notation and some background in convex analysis for later use. We present in Section~\ref{sec:general_case} our main result on the global convergence of the forward-backward algorithm. Section~\ref{sec:application_mfgs} recalls the  ergodic MFG system with local couplings, its finite difference approximation, and its variational formulation. We also compute two dual formulations for which the forward-backward algorithm will be applied in Section~\ref{sec:numerical_results}, devoted to numerical tests and comparisons with other benchmark algorithms in terms of number of iterations and computational time.

\section{Preliminaries}
\label{sec:preliminaries}
Throughout this paper, $\HH$ is a real Hilbert space endowed with the inner product $\scal{\cdot}{\cdot}$ and associated norm $\|\cdot\|$. The weak and the strong convergences are denoted by $\weakly$ and $\to$, respectively. Given $x\in\HH$ and $M>0$, the open and closed ball centered at $x$ with radius $M$ are denoted by $B_{M}(x)$ and $\overline{B}_{M}(x)$, respectively.
Let $f: \HH \to \RX$. The domain of $f$ is $\dom f = \menge{x \in \HH} {f(x) < +\infty}$ and $f$ is proper if $\dom f\neq \emp$. 
Denote by $\Gamma_0(\mathcal{H})$ the class of proper lower semicontinuous convex functions from 
$\HH$ to $]-\infty, +\infty]$. 
Suppose that $f \in \Gamma_0(\HH)$. 
The Fenchel conjugate of $f$ is 
\begin{align}
	\label{eq:conjugate}
	f^*: \HH \to \RX: u \mapsto \sup_{x \in \HH} \big(\scal{x}{u} - f(x)\big).
\end{align} 
We have $f^*\in\Gamma_0(\HH)$ and $f^{**} = f$.
The subdifferential of $f$ is the set-valued operator
\begin{align}
	\label{eq:subdifferential}
	\partial f: \HH \to 2^{\HH} : x \mapsto 
	\menge{u \in \HH}{(\forall y \in \HH) \quad \scal{y-x}{u} + f(x) \leq f(y)}
\end{align}  
and $\dom \partial f=\menge{x\in\HH}{\partial f(x)\neq\emp}$.
The \textit{proximity operator} of $f$ is
\begin{align}
	\label{eq:prox_def}
	\prox_f : \HH \to \HH : x \mapsto \arg\min_{y \in \HH} \left(f(y) + \frac{1}{2} \|x-y\|^2\right),
\end{align}
which, by the Fermat's rule~\cite[Theorem~16.3]{19.Livre1}, is characterized by
\begin{equation}
	\label{eq:prox_char}
	(\forall x \in \HH)(\forall p \in \HH) \quad p = \prox_f x \quad\Leftrightarrow\quad x-p \in \partial f(p).
\end{equation}
For every $\gamma \in\RPP$, ~\cite[Theorem~14.3{\rm(ii)}]{19.Livre1} implies that
\begin{align}
	\label{e:moreau_decomp}
	\prox_{\gamma f} =\Id - \gamma\,\prox_{f^*/\gamma}\circ (\Id/\gamma),
\end{align}
where $\Id\colon\HH\to\HH$ denotes the identity operator. Moreover, by~\cite[Proposition~12.28]{19.Livre1}, $\prox_{\gamma f}$ is firmly nonexpansive, i.e., for all $x,y\in\HH$, 
\begin{equation}
	\label{eq:prox_gamma_f_firm_nonexp}
 \|\prox_{\gamma f}x-\prox_{\gamma f}y\|^{2}+\|(\Id-\prox_{\gamma f})x-(\Id-\prox_{\gamma f})y\|^{2}\leq\|x-y\|^{2}.
\end{equation}
In addition, $f$ is supercoercive if
\begin{equation}
	\label{def:coercive}
	\lim_{|x|\to\pinf}f(x)/|x|=\pinf.
\end{equation}
It follows from~\cite[Proposition~14.15]{19.Livre1} and~\cite[Proposition~18.9]{19.Livre1} that if $f$ is supercoercive and strictly convex, then 
\begin{equation}
	\label{eq:consequence_supercoercive}
	\dom f^*=\HH\:\:\text{and $f^{*}$ is G\^ateaux differentiable on $\HH$.}
\end{equation}

Let $\mathcal{O}\subset\HH$ be an open convex set and let $\beta\in\RPP$. 
Then, $f$ is $\beta$-strongly convex on $\mathcal{O}$ if $f-\beta\|\cdot\|^2/2$ is convex on $\mathcal{O}$.
Now, suppose that $f$ is G\^ateaux differentiable on $\mathcal{O}$.
We say that $\nabla f$ is $\beta$-Lipschitz continuous on $\mathcal{O}$ if 
\begin{equation}
	\label{def:beta_lips_gradient}
	(\forall x\in \mathcal{O})(\forall y\in \mathcal{O})\quad\|\nabla f(x)-\nabla f(y)\|\leq\beta\|x-y\|
\end{equation}
and it is $1/\beta$-cocoercive on $\mathcal{O}$ if 
\begin{equation}
	\label{def:cocoercive_nabla_f}
	(\forall x\in \mathcal{O})(\forall y\in\mathcal{O})\quad\scal{\nabla f(x)-\nabla f(y)}{x-y}\geq \frac{1}{\beta}\|\nabla f(x)-\nabla f(y)\|^{2}.
\end{equation}

A crucial result in our theorem is the following enhanced version of the
Baillon-Haddad theorem~\cite{MR500279}. 
\vspace{0.2cm}
\begin{theorem}{{\rm \cite[Theorem~3.1]{VPA}}}
	\label{t:VPA}
	Let $\mathcal{O}\subset\HH$ be a nonempty convex open set, let $f\colon\mathcal{O}\to\RR$ be a lower semicontinuous convex function, and let $\beta\in\RPP$. Then the following are equivalent:
	\begin{enumerate}[{\rm(i)}]
		\item $f$ is G\^ateaux differentiable in $\mathcal{O}$ and 
		$\nabla f$ is $\beta$-Lipschitz continuous on $\mathcal{O}$.
		\item $\dfrac{\beta}{2}\|\cdot\|^{2}-f$ is convex on $\mathcal{O}$.
		\item $f$ is G\^ateaux differentiable in $\mathcal{O}$ and 
		$\nabla f$ is $1/\beta$-cocoercive on $\mathcal{O}$.
	\end{enumerate}
\end{theorem}
\vspace{0.2cm}

Let $C\subset\HH$ be a nonempty closed convex set. Then 
\begin{equation}
\label{def:c_polar}
C^-=\menge{u\in\HH}{(\forall x\in C)\:\:\scal{x}{u}\le 0}
\end{equation}
is the polar cone to $C$, 
\begin{equation}
	\label{eq:indicator}
	\iota_C: \HH \to \RX: x \mapsto 
	\begin{cases}
		0, &  \text{if }x \in C;\\
		+\infty, &\text{if }x \not\in C   
	\end{cases}   
\end{equation}
is the indicator function of $C$, and
\begin{align}
	\sigma_C : \HH \to\RX: u \mapsto\sup_{x \in C} \scal{x}{u} \label{eq:support}
\end{align}
is the support function of $C$, which satisfies $\sigma_C= (\iota_C)^*$. We also denote by $P_C=\prox_{\iota_C}$ the projection operator onto $C$ and by $d_{C}(x)=\|x-P_Cx\|$ the distance of $x\in\HH$ to $C$.

For further background on convex analysis in Hilbert spaces, the reader is referred to \cite{19.Livre1}. 

\section{Forward-backward algorithm for locally Lipschitz functions}
\label{sec:general_case}
We start with some remarks concerning Problem~\ref{prob:main}.
\vspace{0.2cm}

\begin{remark}
	Consider the setting of Problem~\ref{prob:main}.
	\begin{enumerate}[{\rm(i)}]
		\item Without loss of generality, we 
		assume that, for every $x\in\HH$, 
		\begin{equation}
			0<M'\le M\quad\Rightarrow\quad 
			[L_{M'}(x)\le 
			L_{M}(x)\:\:\text{and}\:\:\mu_{M'}(x)\ge
			\mu_{M}(x)].
		\end{equation}
		
		\item Note that \eqref{e:condgrad} is equivalent
		to the $\mu_M(x)$-strong convexity of $\varphi$ on 
		$B_M(x)$ and the Lipschitz continuity of $\nabla \varphi$ on $B_M(x)$.
		Indeed, observe that {\rm\cite[Proposition~17.7(iii)]{19.Livre1}} implies 
		that the first inequality in \eqref{e:condgrad} is equivalent to 
		the  convexity of 
		$\varphi-\mu_M(x)\|\cdot\|^2/2$ on $B_M(x)$
		and, thus,
		the $\mu_M(x)$-strong convexity of $\varphi$ on 
		$B_M(x)$ {\rm(}in the case $\mu_M(x)=0$ it 
		reduces to convexity{\rm)}. Similarly, the 
		second inequality in \eqref{e:condgrad} is equivalent to the convexity of 
		$L_M(x)\|\cdot\|^2/2-\varphi$ on $B_M(x)$. Hence, 
		it follows from Theorem~\ref{t:VPA} that this is equivalent to the $L_M^{-1}(x)$-cocoercivity of $\nabla \varphi$ on $B_M(x)$ or the Lipschitz continuity of $\nabla \varphi$ on $B_M(x)$.
	\end{enumerate}
\end{remark}
\vspace{0.2cm}

Now we characterize the solution set as fixed points of a suitable operator.
\vspace{0.2cm}
\begin{proposition}
	\label{p:uniq}
	In the context of Problem~\ref{prob:main}, 
	\begin{equation}
		(\forall x^*\in\HH)\quad x^*\in S\:\:
		\Leftrightarrow 
		\:\:(\forall \gamma>0)\quad 
		x^*=\prox_{\gamma\psi}(x^*-\gamma
		\nabla\varphi(x^*)).
	\end{equation}
	Moreover, if
	we have $\mu_M(x^*)>0$ 
	for some $x^*\in S$ and $M>0$,
	then the solution is unique.
\end{proposition}
\begin{proof}
	The first assertion follows from the convexity of $\psi+\varphi$, Fermat's 
	rule, $\dom\varphi=\HH$, and \eqref{eq:prox_char}.
	Now suppose that there exist $x^*\in S$ and $M>0$ such that $\mu_M(x^*)>0$ and that there exists $y^*\in S$ such that $y^*\neq x^*$.
	Then, it follows from the convexity of $\psi+\varphi$ that $z^*=x^*+\frac{M(y^*-x^*)}{2\|y^*-x^*\|}\in B_M(x^*)\cap S$. 
	Moreover, it 
	follows 
	from \cite[Theorem~17.10(iii)]{19.Livre1} that first inequality in 
	\eqref{e:condgrad} implies that $\varphi|_{B_M(x^*)}$ is 
	strictly convex 
	and, thus, $(\psi+\varphi)|_{B_M(x^*)}$ is strictly convex.
	Hence, since $x^*\neq z^*$ are in $S\cap B_M(x^*)$, we obtain a contradiction and the uniqueness follows.
\end{proof}

\begin{proposition}
	\label{p:contractive}
	In the context of Problem~\ref{prob:main}, set
	\begin{equation}
		\label{e:defTgamma}
		(\forall \gamma>0)\quad 
		T_{\gamma}\colon x\mapsto 
		\prox_{\gamma\psi}(x-
		\gamma\nabla \varphi(x)),
	\end{equation}
	let $x\in\HH$, let $M>0$, and let 
	$\gamma\in\left]0,2/L_M(x)\right[$.
	Then, the following hold:
	\begin{enumerate}[{\rm(i)}]
		\item 
		\label{p:contractiveii}
		For all $y,z\in B_M(x)$ we have 
		\begin{align}
			\hspace{-1.8cm}
			\|T_{\gamma}y-T_{\gamma}z\|^2
			&\le \|y-z\|^2-
			\gamma\left(\frac{2}{L_M(x)}-\gamma\right)
			\|\nabla\varphi(y)-\nabla\varphi(z)\|^2\nonumber\\
			&\hspace{1cm}-\|y-T_{\gamma}y
			-(z-T_{\gamma}z)-
			\gamma(\nabla\varphi(y)-\nabla\varphi(z))\|^2.
		\end{align} 
		
		\item 
		\label{p:contractivei}
		Suppose that $\mu_M(x)>0$. 
		Then, for all $y,z\in B_M(x)$, we 
		have
		\begin{equation}
			\|T_{\gamma}y-T_{\gamma}z\|\le 
			\rho_M(x,\gamma)\|y-z\|,
		\end{equation}
		where 
		\begin{equation}
			\label{e:defrho}
			\rho_M\colon(x,\gamma)\mapsto 
			\max\big\{|1-\gamma\mu_M(x)|,|1-\gamma 
			L_M(x)|\big\}\in\left]0,1\right[.
		\end{equation}
	\end{enumerate}
\end{proposition}
\begin{proof}
	Define, for every $\mu\in\left[0,\mu_M(x)\right]$, 
	\begin{equation}
		\phi_{\mu}\colon y\mapsto 
		\varphi(y)-\dfrac{\mu}{2}\|y\|^2.
	\end{equation}
	Since $\varphi$ is G\^ateaux differentiable, then $\phi_{\mu}$ is G\^ateaux differentiable and $\nabla\phi_{\mu}=\nabla\varphi-\mu\Id$. Hence, we 
	obtain from 
	\eqref{e:condgrad} that, for all $y,z\in B_M(x)$,
	\begin{equation}
		\label{e:condgradphi}
		0 \leq \scal{y-z}{\nabla \phi_{\mu}(y)-  
			\nabla \phi_{\mu}(z)} \leq 
		(L_{M}(x)-\mu) \|y-z\|^2.
	\end{equation}
	Therefore, it follows from \cite[Proposition~17.7(iii)]{19.Livre1} that
	$\phi_{\mu}|_{B_M(x)}$ and 
	$\big((L_M(x)-\mu)\|\cdot\|^2/2
	-\phi_{\mu}\big)|_{B_M(x)}$ are convex. Hence, Theorem~\ref{t:VPA} implies that,  for all $y,z\in B_M(x)$, 
	\begin{equation}
		\label{e:cocophi}
		\scal{ 
			y-z}{\nabla\phi_{\mu}(y)-
			\nabla\phi_{\mu}(z)}\ge 
		\dfrac{1}{L_M(x)-\mu}\|\nabla\phi_{\mu}(y)-
		\nabla\phi_{\mu}(z)\|^2.
	\end{equation}
	Now, fix $y$ and $z$ in 
	$B_M(x)$ and  let $\gamma>0$.
	It follows from \eqref{eq:prox_gamma_f_firm_nonexp},
	the identity 
	$\Id-\gamma\nabla\varphi=(1-\gamma\mu)\Id-
	\gamma\nabla\phi_{\mu}$, and \eqref{e:cocophi} that
	\begin{align}
		\hspace{-.9cm}&\|T_{\gamma}y-T_{\gamma}z\|^2\nonumber\\
		&\le 
		\|y-z-\gamma(\nabla\varphi(y)-
		\nabla\varphi(z))\|^2
		-\|y-T_{\gamma}y
		-(z-T_{\gamma}z)
		-\gamma(\nabla\varphi(y)-
		\nabla\varphi(z))\|^2\nonumber\\
		&= 
		\|(1-\gamma\mu)(y-z)-\gamma
		(\nabla\phi_{\mu}(y)-
		\nabla\phi_{\mu}(z))\|^2\nonumber\\
		&\hspace{4cm}-\|y-T_{\gamma}y
		-(z-T_{\gamma}z)
		-\gamma(\nabla\varphi(y)-
		\nabla\varphi(z))\|^2\nonumber\\
		&= 
		(1-\gamma\mu)^2\|y-z\|^2-2\gamma
		(1-\gamma\mu)
		\scal{y-z}{\nabla\phi_{\mu}(y)-
			\nabla\phi_{\mu}(z)}\nonumber\\
		&\hspace{1cm}+\gamma^2\|\nabla\phi_{\mu}(y)-
		\nabla\phi_{\mu}(z)\|^2-\|y
		-T_{\gamma}y
		-(z-T_{\gamma}z)
		-\gamma(\nabla\varphi(y)-
		\nabla\varphi(z))\|^2\nonumber\\
		&\le 
		(1-\gamma\mu)^2\|y-z\|^2\nonumber\\
		&\hspace{1cm}-\big(2\gamma(1-\gamma\mu)
		-\gamma^2(L_M(x)-\mu)\big)
		\scal{y-z}{\nabla\phi_{\mu}(y)-
			\nabla\phi_{\mu}(z)}\nonumber\\
		&\hspace{4cm}-\|y-T_{\gamma}y
		-(z-T_{\gamma}z)
		-\gamma(\nabla\varphi(y)-
		\nabla\varphi(z))\|^2\nonumber\\
		&= 
		(1-\gamma\mu)^2\|y-z\|^2
		-\gamma\big(2-\gamma(L_M(x)+\mu)\big)
		\scal{y-z}{\nabla\phi_{\mu}(y)-
			\nabla\phi_{\mu}(z)}\nonumber\\
		&\hspace{4cm}-\|y-T_{\gamma}y
		-(z-T_{\gamma}z)
		-\gamma(\nabla\varphi(y)-
		\nabla\varphi(z))\|^2.\label{e:auxfin}
	\end{align}
	
	\eqref{p:contractiveii}: Since \eqref{e:auxfin} holds for every 
	$\mu\in \left[0,\mu_M(x)\right]$,
	by setting $\mu=0$ we have 
	$\nabla\phi_{\mu}=\nabla\varphi$
	and the result follows from \eqref{e:cocophi} and \eqref{e:auxfin}.
	
	\eqref{p:contractivei}: We have two cases. 
	\begin{itemize}
		\item Suppose that $\gamma\in\left]0, 
		2/(L_M(x)+\mu)\right]$:
		Since the convexity of $\phi_{\mu}$ implies 
		$\scal{y-z}
		{\nabla\phi_{\mu}(y)-
			\nabla\phi_{\mu}(z)}\ge 0$, we deduce from 
		\eqref{e:auxfin} 
		that
		\begin{equation}\label{1}
			\|T_{\gamma}y-T_{\gamma}z\|\le
			|1-\gamma\mu|\|y-z\|.
		\end{equation}
		\item Suppose that 
		$\gamma\in\left]2/(L_M(x)+\mu)
		,2/L_M(x)\right[$:
		It follows from \eqref{e:condgradphi} and  \eqref{e:auxfin} that
		\begin{align}\label{2}
			\|T_{\gamma}y-T_{\gamma}z\|^2
			&\le 
			\big((1-\gamma\mu)^2+\gamma(L_M(x)
			-\mu)(\gamma 
			(L_M(x)+\mu)-2)\big)
			\|y-z\|^2\nonumber\\
			&=\big(1-\gamma L_M(x)\big)^2\|y-z\|^2.
		\end{align}
	\end{itemize}
	Since previous results hold for every $\mu\in 
	\left[0,\mu_M(x)\right]$, by setting 
	$\mu=\mu_M(x)$, the result follows from \eqref{1} and 
	\eqref{2}.
\end{proof}

Now we state our main result.
\vspace{0.2cm}
\begin{theorem} 
	\label{t:conv}
	In the context of Problem~\ref{prob:main}, let $x_0\in 
	\HH$, let $\varepsilon>0$,
	set $M_0:=\sup_{x^*\in 
		S}\|x^* -x_0\|+\varepsilon$,
	and suppose that 
	$\mathcal{L}_{M_0}:=\sup_{x^*\in 
		S}L_{M_0}(x^*)\in\RPP$.
	Let $\gamma\in  
	\left]0,2/\mathcal{L}_{M_0}\right[$, and
	consider the routine 
	\begin{equation}
		\label{e:FBS}
		(\forall n\in\NN)\quad 
		x_{n+1}=\prox_{\gamma\psi}\big(x_n-
		\gamma\nabla \varphi(x_n)\big).
	\end{equation} 
	Then, the following holds:
	\begin{enumerate}[{\rm(i)}]
		\item 
		\label{t:convi}
		There exists $x^*\in S$ such that  
		$(x_n)_{n\in\NN} \subset 
		B_{M_0}(x^*)$ and
		$x_n\weakly x^*$.
		\item 
		\label{t:convii}
		Suppose that there exists $x^*\in S$ such that
		$\mu_{M_0}(x^*)>0$. 
		Then, $S=\{x^*\}$, $(x_n)_{n\in\NN} \subset B_{M_0}(x^*)$,
		and 
		\begin{equation}\label{t:linear_convergence}
			(\forall n\in\NN)\quad \|x_n-x^*\|\le 
			\big(\rho_{M_0}(x^*,\gamma)\big)^n\|x_0-
			x^*\|,
		\end{equation}
		where $\rho_{M_0}(x^*,\gamma)\in\left]0,1\right[$ is 
		defined 
		in \eqref{e:defrho}.
	\end{enumerate}
	
\end{theorem}
\begin{proof}
	Fix $x^*\in S$. We have 
	$\gamma\in 
	\left]0,2/L_{M_0}(x^*)\right[$ and Proposition~\ref{p:uniq}
	yields $x^*=T_{\gamma}x^*$. Let us first 
	prove by recurrence
	that $(x_n)_{n\in\NN}\subset 
	\overline{B}_{M_0-\varepsilon}(x^*)$. Indeed, since $\|x_0-x^*\|\le\sup_{y^*\in S}\|x_0-y^*\|=M_0-\varepsilon$,
	$x_0\in \overline{B}_{M_0-\varepsilon}(x^*)$. Suppose that 
	$x_n\in \overline{B}_{M_0-\varepsilon}(x^*)$ for some $n\in\NN$.
	Then, it follows from 
	\eqref{e:defTgamma}
	and Proposition~\ref{p:contractive}\eqref{p:contractiveii} that
	\begin{equation}
		\label{e:recurr}
		\|x_{n+1}-x^*\|
		=\|T_{\gamma}x_n-T_{\gamma}x^*\|
		\le\|x_n-x^*\|\le M_0-\varepsilon
	\end{equation}
	and, hence, $x_{n+1}\in \overline{B}_{M_0-\varepsilon}(x^*)$.
	Therefore, we conclude $(x_n)_{n\in\NN}\subset 
	\overline{B}_{M_0-\varepsilon}(x^*)$ and 
	Proposition~\ref{p:contractive}\eqref{p:contractiveii} yields, for every 
	$n\in\NN$, 
	\begin{align}
		\|x_{n+1}-x^*\|^2
		&=\|T_{\gamma}x_n-T_{\gamma}x^*\|^2\nonumber\\
		&\le\|x_n-
		x^*\|^2-
		\gamma\big(2/L_{M_0}(x^*)-\gamma\big)
		\|\nabla\varphi(x_n)-\nabla\varphi(x^*)\|^2
		\nonumber\\
		&\hspace{3cm}-\|x_n-x_{n+1}
		-\gamma(\nabla\varphi(x_n)-
		\nabla\varphi(x^*))\|^2.
	\end{align}
	We deduce from \cite[Lemma~5.31]{19.Livre1} that 
	$(\|x_n-x^*\|)_{n\in\NN}$
	is convergent and that 
	\begin{equation}
		\sum_{n\in\NN}\|\nabla \varphi(x_n)-\nabla 
		\varphi(x^*)\|^2<\pinf
	\end{equation}
and 
\begin{equation}
		\sum_{n\in\NN}\|x_n-x_{n+1}
		-\gamma(\nabla\varphi(x_n)-
		\nabla\varphi(x^*))\|^2<\pinf,
	\end{equation}
	which yields $\nabla \varphi(x_n)\to \nabla 
	\varphi(x^*)$ and 
	\begin{equation}
		\label{e:to0}
		x_n-x_{n+1}
		=(\Id-T_{\gamma})x_n\to 0.
	\end{equation}
	Now, let $x$ be a weak accumulation point of 
	$(x_n)_{n\in\NN}$, say $x_{k_n}\weakly 
	x$. Since  $\overline{B}_{M_0-\varepsilon}(x^*)$ is weakly 
	closed, $x$ and $(x_{k_n})_{n\in\NN}$
	are in $\overline{B}_{M_0-\varepsilon}(x^*)$. Hence, 
	it follows from the nonexpansiveness of 
	$T_{\gamma}$ in $B_{M_0}(x^*)$, guaranteed by 
	Proposition~\ref{p:contractive}\eqref{p:contractiveii}, that
	\begin{align}
		\label{e:auxopial}
		&\|x-T_{\gamma}x\|^2\nonumber\\
		=&\|x_{k_n}-T_{\gamma}x\|^2
		-\|x_{k_n}-x\|^2-
		2\scal{x_{k_n}-x}{x-
			T_{\gamma}x}\nonumber\\
		=&\|x_{k_n}-T_{\gamma}x_{k_n}\|^2
		+\|T_{\gamma}x_{k_n}-T_{\gamma}x\|^2
		+2\scal{x_{k_n}-
			T_{\gamma}x_{k_n}}
		{T_{\gamma}x_{k_n}-T_{\gamma}x}\nonumber\\
		&\hspace{2.8cm}-\|x_{k_n}-x\|^2-
		2\scal{x_{k_n}-x}{x-
			T_{\gamma}x}\nonumber\\
		\le& \|(\Id-T_{\gamma})x_{k_n}\|^2
		+2\scal{(\Id-T_{\gamma})x_{k_n}}
		{x_{k_n+1}-T_{\gamma}x}-
		2\scal{x_{k_n}-x}{x-
			T_{\gamma}x}.
	\end{align}
	Therefore, \eqref{e:to0}, $x_{k_n}\weakly 
	x$, and boundedness of 
	$(x_{k_n})_{n\in\NN}$ imply that the right hand side 
	of \eqref{e:auxopial} tends to $0$ as $n\to\pinf$, which yields 
	$x=T_{\gamma}x$ and, thus, 
	$x\in S$, in view of Proposition~\ref{p:uniq}. Then, 
	\eqref{t:convi} follows from 
	\cite[Lemma~2.47]{19.Livre1}.
	
	\eqref{t:convii}: The uniqueness of the solution $x^*$ 
	follows from Proposition~\ref{p:uniq}. Since 
	$\mu_{M_0}(x^*)>0$, 
	Proposition~\ref{p:contractive}\eqref{p:contractivei} implies that 
	$(x_n)_{n\in\NN}\subset B_{M_0}(x^*)$,
	following the same argument used in \ref{t:convi}. Therefore,
	for every $n\in\NN$,
	\begin{equation}
		\|x_{n+1}-x^*\|
		=\|T_{\gamma}x_n-T_{\gamma}x^*\|
		\le\rho_{M_0}(x^*,\gamma)\|x_n-
		x^*\|\le
		\big(\rho_{M_0}(x^*,\gamma)\big)^{n+1}
		\|x_0-x^*\|
	\end{equation}
	and the proof is complete.
\end{proof}
\vspace{0.2cm}

\begin{remark}\phantom{ }
	\label{rem:optrate}
	\begin{enumerate}[{\rm(i)}]
		\item 
		\label{rem:optratei}
		Note that, for every $x^*$ and $M>0$, 
		$\rho_M(x^*,\cdot)$ is decreasing in 
		$\left]0,2/(L_M(x^*)+\mu_M(x^*))\right]$
		and increasing in 
		$\left]2/(L_M(x^*)+\mu_M(x^*)),
		2/L_M(x^*)\right[$. Therefore,
		the optimal convergence rate is obtained by choosing 
		$\gamma^*=2/(L_M(x^*)+\mu_M(x^*))$, 
		which yields
		\begin{equation}
			\rho_M(x^*,\gamma^*)=\dfrac{L_M(x^*)-\mu_M(x^*)}
			{L_M(x^*)+\mu_M(x^*)}.
		\end{equation}
		
		\item In the particular case when $\nabla\varphi$ is globally Lipschitz continuous with constant $L>0$, we have, for every $M>0$, $L_M(\cdot)\equiv\mathcal{L}_{M}=L$ and we recover the convergence of the classical FBS {\rm(}see, e.g., {\rm\cite[Theorem~3.4]{mms05}}{\rm)}. In addition, if $\varphi$ is $\rho$-strongly convex in $\HH$ for some $\rho>0$, we recover the linear convergence rate in 
{\rm\cite{Nelly,Taylor}}.
		
		\item In general, $M_0$ and $\mathcal{L}_{M_0}$ in 
		Theorem~\ref{t:conv} are difficult to compute exactly since the 
		solution set $S$ is not known. However,
		it is possible to over-estimate them by knowing a priori bounds on the 
		solutions. Indeed, if $S\subset U$, where $U\subset\HH$ is an a 
		priori region in which the solutions are known to be, we have
		\begin{equation}
			M_0\le M:=\sup_{x\in 
				U}\|x-x_0\|\quad \text{and}\quad 
			\mathcal{L}_{M_0}\le \sup_{x\in U}L_{M}(x).
		\end{equation}
		In the locally strongly convex case, if the unique solution 
		$x^*$ is 
		known to be in an a priori set $U$, we can also under-estimate 
		$\mu_M(x^*)$ by $\inf_{x\in 
			U}\mu_M(x)$ provided that the latter is strictly positive.
	\end{enumerate}
\end{remark}
\section{Application to ergodic variational mean field games}
\label{sec:application_mfgs}
Consider the following stationary MFGs system with monotone local couplings
\begin{equation}
	\label{MFG_ergodic_system}
	\begin{array}{rll}
		- \nu \Delta u + \mathscr{H}(\nabla u) + \lambda \; 
		&= f(x,m(x))  &\text{in } \TT^d, \\[5pt]
		- \nu \Delta m - \text{div}\left(\nabla \mathscr{H}(\nabla 
		u)m\right) \; \; &=0 \qquad &\text{in } \TT^d, \\[5pt]
		\int_{\TT^d} u(x) {\rm d} x=0, \:\:  \int_{\TT^d}m(x){\rm d} x=1,
		& m>0.
	\end{array}
\end{equation}

Here, $\nu>0$, $\TT^{d}$ denotes the $d$-dimensional torus, $u,\,m\colon\TT^{d}\to\RR$ and $\lambda\in \RR$ are the unknowns,  the {\it Hamiltonian}  $\mathscr{H}\colon\RR^{d}\to\RR$ is convex and differentiable, and $\TT^{d}\times\RPP\ni(x,\rho)\mapsto f(x,\rho)\in\RR$ is continuous and strictly increasing with respect to $r$. Moreover, we suppose that
\begin{align}
	(\forall x\in\TT^{d})(\forall \rho\in\RPP)\quad\int_{0}^{\rho}f(x,\rho')\,{\rm d} \rho'<\infty,
	\label{eq:integrability_condition}\\
	(\forall x\in\TT^{d})\quad\lim_{\rho\to\pinf}f(x,\rho)=\pinf.
	\label{eq:f_unbounded}
\end{align}

Under suitable assumptions on the growth of $\mathscr{H}$ and on the regularity of $f$, system~\eqref{MFG_ergodic_system} corresponds to the optimality system of a convex variational problem (see e.g.~\cite{Lasry_Lions_2007,meszaros_silva_2018})  and admits a unique smooth solution $(u^{*},m^{*},\lambda^{*})$ (see e.g.~\cite{MR2928380,MR3333058,MR3415027,MR3160525,MR2928381,MR3623401}). 

In this section, we consider the finite difference scheme proposed in~\cite{Achdou_Capuzzo_Dolcetta_2010} to approximate the solution to~\eqref{MFG_ergodic_system} in the two-dimensional case $d=2$ and when 
\begin{equation}
	\label{hyp:assumption_H}
	(\forall\,p\in\RR^{d})\quad\mathscr{H}(p)=\ell^{*}(\|p\|),
\end{equation}
where  $\ell\in\Gamma_{0}(\RR)$ is non-negative, supercoercive, increasing, and strictly convex on its domain $\dom \ell=\RP$. The non-negativity of $\ell$ over its domain yields that $\ell^{*}$ is increasing. Moreover, by~\eqref{eq:consequence_supercoercive} and~\cite[Corollary~17.44]{19.Livre1}, the supercoercivity and the strict convexity of $\ell$ imply that $\dom \ell^*=\RR$ and that $\ell^{*}$ is differentiable on $\RR$.
\vspace{0.2cm}
\begin{remark} Let $\gamma\in[1,\pinf[$ and set $\gamma'=\gamma/(\gamma-1)$. 
	A typical example of a Hamiltonian satisfying the previous assumptions is given by $\mathscr{H}(p)=\|p\|^{\gamma'}/\gamma'$ for all $p\in\RR^{d}$, which is obtained from~\eqref{hyp:assumption_H} with $\ell(r)=|r|^{\gamma}/\gamma+\iota_{\RP}(r)$ for all $r\in\RR$.  
\end{remark}

Let $N\in\NN$, set $h=1/N$ and set $\TT_{h}=\{x_{i,j}=(ih,jh)\,|\,i,\,j=0,\hdots,N-1\}$. Let $\M_{h}$ be the set of real-valued functions defined on $\TT_{h}$, set $\mathcal{W}_h= 
\M_h^4$, and set $\Y_h=\{z\in\M_{h}\,|\,
\sum_{i,j=0}^{N-1}z(x_{i,j})=0\}$. For notational simplicity, given $z\in\M_{h}$, we will write $z_{i,j}$ for $z(x_{i,j})$. The discrete differential operators $D_1\colon \mathcal{M}_h\to\mathcal{Y}_h$,
$D_2\colon \mathcal{M}_h\to\mathcal{Y}_h$, $D_h\colon 
\mathcal{M}_h\to\mathcal{Y}_h^4$, $\Delta_h\colon 
\mathcal{M}_h\to\mathcal{Y}_h$, and 
$\mbox{div}_h\colon\mathcal{W}_h\to 
\mathcal{Y}_h$ are defined as
\begin{align}
	(D_1 z)_{i,j}&= \frac{z_{i+1,j}-z_{i,j}}{h}, \quad (D_2 
	z)_{i,j}=\frac{z_{i,j+1}-z_{i,j}}{h},\nonumber\\[4pt]
	[D_hz]_{i,j}&=\big((D_1 z)_{i,j}, (D_1 z)_{i-1,j}, (D_2 z)_{i, j}, (D_2 
	z)_{i,j-1}\big),\nonumber\\[4pt]
	(\Delta_h z)_{i,j}&= \frac{z_{i-1,j} 
		+z_{i+1,j}+z_{i,j-1}+z_{i,j+1}-4z_{i,j}}{h^2},\nonumber\\[4pt]
	(\text{div}_{h}(w))_{i,j}&= (D_1 w^1)_{i-1, j} + 
	(D_1w^2)_{i,j}+(D_2w^3)_{i,j-1}+ (D_2w^4)_{i,j},\nonumber
\end{align}
for all  $z\in\mathcal{M}_h$, $w\in\mathcal{W}_h$, $i$, $j=0,\hdots, N-1$, and the  sums between the indexes are taken modulo $N$.  

Let us set $\mathbb{R}_+=\RP$ and  $\mathbb{R}_-=\RM$. 
Let $K:=\mathbb{R}_+\times \mathbb{R}_- \times \mathbb{R}_+ 
\times \mathbb{R}_-\subset \RR^4$, denote by $P_{K}(p)$ the euclidean projection of $p\in\RR^{4}$ onto the closed and convex cone $K$, and set 
\begin{equation}
	\label{def:kappa}
	(\forall C\in\{K,\RR^{4}\})\quad
	\zeta_{C}\colon\RR^{4}\to\RR^{4}\colon\xi\mapsto 
	\begin{cases}
		\displaystyle \frac{{\ell^*}'(\|P_C\xi\|)}{\|P_C\xi\|}P_C\xi,&\text{if}\:\: P_C\xi\ne 0;\\
		0,&\text{if}\:\: P_C\xi= 0.
	\end{cases}
\end{equation}

The finite difference discretization of system~\eqref{MFG_ergodic_system} reads as follows:
\begin{equation}
	\label{eq:finite_difference_scheme}
	\begin{array}{l}
		\displaystyle-\nu(\Delta_h u)_{i,j} +\ell^{*}\big(\|P_{K}\big(-[D_hu]_{i,j}\big)\|\big)+\lambda=  
		f(x_{i,j},m_{i,j}),\quad\text{for all }i,\,j=0,\hdots,N-1,\\[14pt]
		\displaystyle-\nu(\Delta_h m)_{i,j} +\displaystyle\big(\text{div}_h\big(m\,\zeta_{K}(-[D_hu ])\big)\big)_{i,j}=0,\quad\text{for all }i,\,j=0,\hdots,N-1, \\[10pt]
		\displaystyle\sum_{i,j=0}^{N-1}u_{i,j} =0,\;\; h^2 \sum_{i,j=0}^{N-1}m_{i,j} =1,\;\;m_{i,j} > 0,\;\;\;\text{for all }i,\,j=0,\hdots,N-1.
	\end{array}
\end{equation}

As for~\eqref{MFG_ergodic_system}, system~\eqref{eq:finite_difference_scheme} corresponds to the optimality condition for the solution to a convex variational problem. In order to define this problem, set 
\begin{equation}
	b\colon\RR\times\RR^4\to\RX\colon (\rho,\omega)\mapsto
	\begin{cases}
		\rho\,\funp\left(\dfrac{\|\omega\|}{\rho}\right),\quad&\text{if 
		}\rho>0;\\
		0,&\text{if }(\rho,\omega)=(0,0);\\
		\pinf&\text{otherwise}
	\end{cases} 
\end{equation} 
and, for every $i$, $j=0,\hdots,N-1$, define 
\begin{equation}
	\label{def:F_ij}
	F_{i,j}(\rho)=\begin{cases}
		\displaystyle
		\int_{0}^{\rho}f(x_{i,j},\rho') {\rm d}\rho', &\text{if } \rho\in\RP;\\
		\pinf, &\text{otherwise}.
	\end{cases}
\end{equation}
Notice that~\eqref{eq:integrability_condition} yields $\dom F_{i,j}=\RP$, and, since $f(x_{i,j},\cdot)$ is strictly increasing, the function $F_{i,j}$ is strictly convex. Moreover, it follows from~\eqref{eq:integrability_condition},~\eqref{eq:f_unbounded}, and~\eqref{def:F_ij} that $F_{i,j}$ is supercoercive. Thus, as for $\ell$, we have that $\dom F_{i,j}^{*}=\RR$ and $F_{i,j}^{*}$ is increasing and differentiable on $\RR$. 
\vspace{0.2cm}
\begin{problem}
	\label{prob:MFG}
	The problem is to
	\begin{align}
		&\min_{(m,w)\in\M_h\times \mathcal{W}_h}  
		\sum_{i,j=0}^{N-1}\Big(b(m_{i,j},w_{i,j})+F_{i,j}(m_{i,j})\Big)\nonumber\\
		&\qquad \;  \text{s.t.} \quad 
		\begin{cases}
			-\nu\Delta_hm + \mbox{{\rm div}}_h w=0,\\[4pt]
			\displaystyle h^2\sum_{i,j=0}^{N-1}m_{i,j}= 1,\\[10pt]
			w_{i,j}\in K,\quad \text{for all }i,j=0,\hdots,N-1.
		\end{cases}
		\label{general_discrete_formulation}
	\end{align}
\end{problem}

Note that, in view of~\cite{Rocky_66}, $b\in\Gamma_{0}(\RR\times\RR^{4})$ and hence Problem 4.2 is convex.
\vspace{0.2cm}

\begin{proposition} 
	\label{prop:optimality_condition_discrete_problem}
	System~\eqref{eq:finite_difference_scheme} has a unique solution $(m^{h},u^{h},\lambda^{h})$. Moreover, setting
	\begin{equation}
		\label{eq:optimal_w_h}
		w^{h}=m^{h}\zeta_{K}(-[D_hu^{h} ]),
	\end{equation}
	we have that $(m^{h},w^{h})$ is the unique solution to Problem~\ref{prob:MFG}. 
\end{proposition}
\begin{proof}
	When $\mathscr{H}=\|\cdot\|^{2}/2$, the proof of the previous result is given in~\cite[Proposition~4.1]{MR4534446}. The case when $\mathscr{H}$ is given by~\eqref{hyp:assumption_H} follows from exactly the same arguments. 
\end{proof}

In what follows, our aim will be to apply the results in Section~\ref{sec:general_case} to a dual formulation of Problem~\ref{prob:MFG}. For this purpose, the following result will play an important role. 
\vspace{0.2cm}
\begin{proposition}
	\label{p:0}
	In the context of Problem~\ref{prob:MFG}, let 
	$C\in\{K,\RR^4\}$ and, for every $i,j\in\{0,\ldots,N-1\}$,
	define
	\begin{equation}
		\phi_{i,j}\colon 
		\RR\times\RR^4\to\RX\colon (\rho,\omega)
		\mapsto b(\rho,\omega)+F_{i,j}(\rho)+\iota_{C}(\omega).
	\end{equation}
	Then, for every $(\rho^*,\omega^*)\in \RR\times\RR^4$, the following hold:
	\begin{enumerate}[{\rm(i)}]
		\item 
		\label{p:0i}
		$
		\phi_{i,j}^*
		(\rho^*,\omega^*)=
		F_{i,j}^*\big(\rho^*+\funp^*(\|P_C\omega^*\|)\big).
		$
		\item 
		\label{p:0ii}
		$\phi_{i,j}^*$ is differentiable in $\RR\times\RR^4$ and 
		\begin{equation}
			\label{eq:nabla_phi_etoile}
			\nabla\phi^*_{i,j}(\rho^*,\omega^*)
			={F_{i,j}^*}'(\rho^*+\funp^*(\|P_{C}\omega^*\|))
			\begin{pmatrix}1\\\zeta_{C}(\omega^{*})
			\end{pmatrix},
		\end{equation}
		where $\zeta_{C}$ is defined in~\eqref{def:kappa}.
	\end{enumerate}
\end{proposition}

\begin{proof}
	Fix 
	$i$ and $j$ in $\{0,\ldots,N-1\}$ and $(\rho^*,\omega^*)\in 
	\RR\times\RR^4$.
	
	\eqref{p:0i}: Set $\widehat{\rho}=1$ and $\widehat{\omega}=(1,-1,1,-1)$. Since $b$ is continuous at $(\widehat{\rho},\widehat{\omega})$ and $F_{i,j}(\widehat{\rho})+\iota_{C}(\widehat{\omega})<\pinf$, it follows from \cite[Proposition~15.5(iv) \& Theorem~15.3]{19.Livre1}, \cite[Proposition~13.30 \& Example~13.3(i)-(ii)]{19.Livre1}, and \cite[Proposition~2.3(iv)]{Comb18} that 
	\begin{align}
		\phi_{i,j}^*(\rho^*,\omega^*)
		&=\inf_{(\rho,\omega)\in\RR\times \RR^4}b^*(\rho^*-\rho,\omega^*-\omega)+F_{i,j}^*(\rho)+\iota_{C^-}(\omega)\nonumber\\
		&=\inf_{\substack{(\rho,\omega)\in\RR\times C^{-},\\
				\rho^*+\ell^*(\|\omega^*-\omega\|)\le \rho}}F_{i,j}^*(\rho)\nonumber\\
		&=\inf\limits_{\omega\in C^{-}}F_{i,j}^*(\rho^*+\ell^*(\|\omega^*-\omega\|))\nonumber\\
		&=F_{i,j}^*(\rho^*+\ell^*(\|\omega^*-P_{C^-}\omega^*\|)),\nonumber
	\end{align}
	where, in the last equality, we have used that $\RR\ni t\mapsto F_{i,j}^*(\rho^*+\ell^*(t))\in\RR$ is increasing. The result follows from the identity $P_C=\Id-P_{C^-}$.
	
	\eqref{p:0ii}:
	Using that, for all $\xi\in\RR^{4}$, $\|P_{C}\xi\|=d_{C^-}(\xi)$, ${\ell^*}'(0)=0$, and $P_C\xi=0$ if and only if $\xi\in C^{-}$, it follows from \cite[Example 17.33]{19.Livre1} that $\ell^*\circ \|\cdot\|\circ P_C$ is differentiable and $\nabla(\ell^*\circ \|\cdot\|\circ P_C)=\zeta_{C}$. Thus, by~\ref{p:0i}, $\phi_{i,j}^*$ is differentiable in $\RR\times\RR^4$ and~\eqref{eq:nabla_phi_etoile} holds. 
\end{proof}

For every $C\in\{K,\RR^4\}$, we define
\begin{multline}
	D_C=\Big\{(m,w)\in\M_{h}\times \mathcal{W}_h\,\Big|\,-\nu\Delta_{h}m + \mbox{div}_h w=0,\\
	h^2\sum_{i,j=0}^{N-1}m_{i,j}= 1,\:\:\big(\forall i,j\in\{0,\hdots,N-1\}\big)\quad w_{i,j}\in C\Big\}.
\end{multline}
Note that Problem~\ref{prob:MFG} can be written equivalently as
\begin{equation}
	\label{e:form1}
	\min_{(m,w)\in 
		D_{C_1}}\sum_{i,j=0}^{N-1}\big(b(m_{i,j},w_{i,j})+F_{i,j}(m_{i,j})+\iota_{C_2}(w_{i,j})\big)
\end{equation}
either if $(C_1,C_2)=(K,\RR^4)$ or $(C_1,C_2)=(\RR^4,K)$.
In view of \cite[Proposition~13.30 \& 
Definition~15.19]{19.Livre1} and Proposition~\ref{p:0}, formulation
\eqref{e:form1} leads to the following 
Fenchel-Rockafellar dual problem
\begin{align}
	&\min_{(\theta,v)\in \M_{h}\times\mathcal{W}_{h}}
	\sigma_{D_{C_1}}(-\theta,-v)+\sum_{i,j=0}^{N-1}
	F_{i,j}^*\big(\theta_{i,j}+\funp^*(\|P_{C_2}v_{i,j}\|)\big).
	\label{e:form1d}
\end{align}

Note that \eqref{e:form1d} can be written as \eqref{e:mainprob}, where
$\psi=\sigma_{D_{C_1}}\circ(-\Id)$, $\varphi\colon (\theta,v)\mapsto \sum_{i,j=0}^{N-1}\phi^*_{i,j}(\theta_{i,j},v_{i,j})$, and, for every $i,j\in\{0,\ldots,N-1\}$, $\phi^*_{i,j}\colon (\theta_{i,j},v_{i,j})\mapsto F_{i,j}^*(\theta_{i,j}+\funp^*(\|P_{C_2}v_{i,j}\|))$. By Proposition~\ref{p:0}\eqref{p:0ii} we deduce that $\varphi$ is differentiable. Under additional assumptions on ${\ell^*}'$ and $({F_{i,j}^*}')_{0\le i,j\le N-1}$, one can prove that \eqref{e:form1d} is a particular instance of Problem~\ref{prob:main}. In the next section, we provide some examples and explicit computations of Lipschitz and strong convexity constants for the function $\varphi$. 

Notice that, by~\cite[Proposition~19.4]{19.Livre1} and Proposition~\ref{p:0}\eqref{p:0ii}, we recover the unique solution $(m^{h},w^{h})$ to Problem~\ref{prob:MFG} from a solution $(\theta^{h},v^{h})$ to Problem~\ref{e:form1d} through the expressions 
\begin{equation}
	\label{eq:primal_dual_sol}
	(\forall\,i,j\in\{0,\hdots,N-1\})\quad m_{i,j}^{h}={F_{i,j}^*}'\Big(\theta^{h}_{i,j}+\funp^*(\|P_{C}v^{h}_{i,j}\|)\Big)\:\:\text{and}\:\: w_{i,j}^{h}=m_{i,j}^{h}\,\zeta_{C}(v^{h}_{i,j}).
\end{equation}
Since~\eqref{eq:optimal_w_h} implies that  
\begin{equation} 
	\label{eq:non_linear_term}
	(\forall\,i,j\in\{0,\hdots,N-1\})\quad\ell^{*}\big(\|P_{K}\big(-[D_hu]_{i,j}\big)\|\big)= \ell^{*}\big(\ell'\big(\|w_{i,j}^{h}\|/m_{i,j}^{h}\big)\big),
\end{equation}
one can compute $(u^{h},\lambda^{h})$ by solving the linear system consisting in the first equation in~\eqref{eq:finite_difference_scheme} together with the condition $\sum_{i,j=0}^{N-1}u^{h}_{i,j} =0$.

Therefore, in what follows, we focus on approximating the solution $(\theta^{h},v^{h})$ to~\eqref{e:form1d} for which we will consider the following routine.
\vspace{0.2cm}
\begin{algo}
	\label{alg:FBS1d}
	Let $(\theta_0,v_0)\in \M_{h}\times\mathcal{W}_{h}$, 
	let $\gamma\in\RPP$, and consider the recurrence:
	\begin{equation}
		\label{e:FBS_form2d}
		(\forall n\in\NN)\quad 
		\begin{array}{l}
			\left\lfloor
			\begin{array}{l}
				\text{for }i,j=0,\ldots,N-1\\
				\left\lfloor
				\begin{array}{l}
					t_{i,j,n} = \theta_{i,j,n}-\gamma {F_{i,j}^*}'(\theta_{i,j,n}+\funp^*(\|P_{C_2}v_{i,j,n}\|)) \\
					\nu_{i,j,n}=v_{i,j,n}-\gamma {F_{i,j}^*}'(\theta_{i,j,n}+\funp^*(\|P_{C_2}v_{i,j,n}\|))\zeta_{C_{2}}(v_{i,j,n})\\
				\end{array}
				\right.\\
				(\theta_{n+1},v_{n+1})=\prox_{\gamma\sigma_{D_{C_1}}}((t_{i,j,n},\nu_{i,j,n})_{0\le i,j\le N-1}).
			\end{array}
			\right.
		\end{array}
	\end{equation}
\end{algo}

\section{Numerical results}
\label{sec:numerical_results}
In this section, we consider some numerical experiments in two instances. In the first experiment, we consider the first order mean field game studied in {\cite{Gomesex}, in which the exact solution is known. In this context we compare the relative exact error of Algorithm~\ref{alg:FBS1d} when $(C_1,C_2)=(K,\RR^4)$ and when $(C_1,C_2)=(\RR^4,K)$. The second experiment is devoted to a second-order mean field game in which the coupling term is the sum of a power function with an entropic penalization. We compare state-of-the art algorithms to both algorithms tested in the first experiment.
	We test a bunch of different values of time steps and record the best performance corresponding to the minimum number of iterations of each algorithm considered in each numerical comparison.  
	
	\subsection{A first order mean field game with a logarithmic coupling}
	\label{subsec:ex1}
	We consider here the MFGs system
	\begin{align}
		\dfrac{1}{2}|\nabla u|^{2}+\lambda&=\log(m)-\sin(2\pi x)-\sin(2\pi y)\quad\text{in }\TT^{2},\\
		\text{div}(m\nabla u)&=0\quad\text{in }\TT^{2},\;\;\int_{\TT^{2}}u(x){\rm d}x=0,\;\;\int_{\TT^{2}}m(x){\rm d}x=1,\;\; m>0.
	\end{align}
	Its finite difference discretization~\eqref{eq:finite_difference_scheme} is given by 
	\begin{equation}
		\label{eq:finite_difference_scheme_example_1}
		\begin{array}{l}
			\frac{1}{2}\|P_{K}\big(-[D_hu]_{i,j}\big)\|^{2}+\lambda=  
			\log(m_{i,j})-\sin(2\pi ih)-\sin(2\pi jh),\\[6pt]
			\hspace{8cm} \text{for all } i,j=0,\hdots,N-1,\\[2pt]
			\displaystyle\big(\text{div}_h\big(m\,P_{K}\big(-[D_hu]\big)\big)\big)_{i,j}=0,\quad\text{for all } i,j=0,\hdots,N-1, \\[6pt]
			\displaystyle\sum_{i,j=0}^{N-1}u_{i,j} =0,\;\; h^2 \sum_{i,j=0}^{N-1}m_{i,j} =1,\;\;m_{i,j}>0,\;\;\;\text{for all } i, j=0,\hdots,N-1.
		\end{array}
	\end{equation}
	Note that, since $\nu=0$ in \eqref{eq:finite_difference_scheme_example_1}, one cannot apply Proposition~\ref{prop:optimality_condition_discrete_problem}. However, one can easily check (see~\cite[Remark 2.1]{BAKS}) that~\eqref{prop:optimality_condition_discrete_problem} corresponds to the optimality condition of Problem~\ref{prob:MFG}, where $\nu=0$, $\funp=|\cdot|^2/2+\iota_{\RP}$, for every $i,j\in\{0,\ldots,N-1\}$, 
	$F_{i,j}\colon \rho\mapsto \rho(\ln(\rho)-1)-c_{i,j}\rho$ and 
	$c_{i,j}= \sin(2\pi ih)+\sin(2\pi jh)$. One of the main interests of system~\eqref{prop:optimality_condition_discrete_problem} is that it admits an explicit solution (see~\cite{Gomesex}) which allows us to establish the precise error for its approximation. In particular, the explicit solution to the dual problem \eqref{e:form1d} in this case 
	is 
	\begin{equation}
		\label{e:solgomes}
		\theta^*_{i,j}=t^*,\quad\text{for all }i,\,j=0,\hdots,N-1,\;\;\text{where } t^*=-\ln\bigg(h^2\sum_{i,j=0}^{N-1}e^{c_{i,j}}\bigg),\;\;\text{and }v^{*}=0.
	\end{equation}
	Moreover, since $\funp^*=\max\{\text{sgn}(\cdot)|\cdot|^2/2,0\}$ and
	$F^*_{i,j}\colon \rho^*\mapsto e^{\rho^*+c_{i,j}}$, 
	the dual formulation~\eqref{e:form1d}
	reduces to 
	\begin{equation}
		\label{e:form1dex}
		\min_{(\theta,v)\in\M_{h}\times\mathcal{W}_{h}}
		\sigma_{D_{C_1}}(-\theta,-v)+	 \sum_{i,j=0}^{N-1}
		e^{\theta_{i,j}+\frac{1}{2}\|P_{C_2}v_{i,j}\|^2+c_{i,j}}\;.
	\end{equation}
	Now, in the case when $C_1=K$ and $C_2=\RR^4$, let us prove that \eqref{e:form1dex} is a particular instance of Problem~\ref{prob:main} which is also strongly convex. Indeed, let $M_0>0$, 
	for every $i,j\in\{0,\ldots,N-1\}$, set 
	\begin{equation}
		L_{M_0,i,j}(\theta_{i,j}^*,v_{i,j}^*)=e^{\theta_{i,j}^*+c_{i,j}}g(M_0^2)\le 
		e^{\theta_{i,j}^*+c_{i,j}+M_0}(3+M_0^2)
\end{equation} 
and 
	\begin{equation}
		\mu_{M_0,i,j}(\theta_{i,j}^*,v_{i,j}^*)=e^{\theta_{i,j}^*+c_{i,j}-M_0},
	\end{equation}
	where $g\colon M\mapsto\max_{x\in[0,M]}e^{\sqrt{M-x}}(1+x/2+(2+x/2)\sqrt{x/(4+x)})$.
	Then, $\phi_{i,j}^*:(\rho^*, \omega^*)\mapsto e^{\rho^*+\|\omega^*\|^2/2+c_{i,j}}$ is 
	$\mu_{M_0,{i,j}}(\theta_{i,j}^*,v_{i,j}^*)$-strongly convex, differentiable in $B((\theta_{i,j}^*,v_{i,j}^*),M_0)$. Moreover, in view of Proposition~\ref{p:0}\eqref{p:0ii},
	${\ell^*}'=\max\{\Id,0\}$, $\zeta_{C_{2}}=\Id$, and
	$\nabla\phi^*_{i,j}\colon (\rho^*,\omega^*)\mapsto 
	e^{\rho^*+\|\omega^*\|^2/2+c_{i,j}}(1,\omega^*)^{\top}$ 
	is $L_{M_0,{i,j}}(\theta_{i,j}^*,v_{i,j}^*)$-Lipschitz 
	continuous in $B((\theta_{i,j}^*,v_{i,j}^*),M_0)$.
	Therefore, $\varphi\colon (\theta,v)\mapsto\sum_{i,j=0}^{N-1}\phi^*_{i,j}(\theta_{i,j},v_{i,j})$ is 
	$\mu_{M_0}(\theta^*,v^*)$-strongly convex and $\nabla\varphi$ is $L_{M_0}(\theta^*,v^*)$-Lipschitz 
	continuous in $B((\theta^*,v^*),M_0)$, where
	\begin{multline}
		L_{M_0}(\theta^*,v^*)=\max_{i,j=0,\ldots,N-1}L_{M_0,i,j}(\theta_{i,j}^*,v_{i,j}^*)=
		e^{t^*+\overline{c}}g(M_0^2)\le 
		e^{t^*+\overline{c}+M_0}(3+M_0^2),\\
		\mu_{M_0}(\theta^*,v^*)=
		\min_{i,j=0,\ldots,N-1}\mu_{M_0,i,j}(\theta_{i,j}^*,v_{i,j}^*)
		=e^{t^*+\underline{c}-M_0}>0,
	\end{multline}
	$\underline{c}=\min_{i,j=0,\ldots,N-1}c_{i,j}$, and $\overline{c}=\max_{i,j=0,\ldots,N-1}c_{i,j}$.
	In this case, Algorithm~\ref{alg:FBS1d} reduces to
	\begin{equation}
		\label{e:FBS_form1dex}\tag{DFB0}
		(\forall n\in\NN)\quad 
		\begin{array}{l}
			\left\lfloor
			\begin{array}{l}
				\text{for }i,j=0,\ldots,N-1\\
				\left\lfloor
				\begin{array}{l}
					t_{i,j,n} = \widetilde{\theta}_{i,j,n}-\gamma e^{\widetilde{\theta}_{i,j,n}+\|\widetilde{v}_{i,j,n}\|^2/2+c_{i,j}}\\
					\nu_{i,j,n}=(1-\gamma e^{\widetilde{\theta}_{i,j,n}+\|\widetilde{v}_{i,j,n}\|^2/2+c_{i,j}})\widetilde{v}_{i,j,n}\\
				\end{array}
				\right.\\
				(\widetilde{\theta}_{n+1},\widetilde{v}_{n+1})=\prox_{\gamma\sigma_{D_K}}((t_{i,j,n},\nu_{i,j,n})_{0\le i,j\le N-1}),
			\end{array}
			\right.
		\end{array}
	\end{equation}
	where $(\widetilde{\theta}_{0},\widetilde{v}_{0})\in\M_{h}\times\mathcal{W}_{h}$ and $\gamma >0$. In this context,
	if $\gamma\in\left]0,2/L_{M_0}(\theta^*,v^*)\right[$, with $M_{0}=\|(\widetilde{\theta}_{0},\widetilde{v}_{0})-({\theta}^{*},{v}^{*})\|$, Theorem~\ref{t:conv}\eqref{t:convii} guarantees 
	the linear convergence of the sequence $(\widetilde{\theta}_n,\widetilde{v}_n)_{n\in\NN}$ generated by \eqref{e:FBS_form1dex}
	to $(\theta^*,v^*)$. In addition, the optimal rate is achieved by choosing $\gamma^*$ from Remark~\ref{rem:optrate}\eqref{rem:optratei}. Observe that, in view of \eqref{e:moreau_decomp}, the computation of $\prox_{\gamma\sigma_{D_K}}$ is obtained from the projection onto $D_K$, which needs a subroutine.
	
	On the other hand, if $C_1=\RR^4$ and $C_2=K$,
	$\varphi\colon (\theta,v)\mapsto \sum_{i,j=0}^{N-1}e^{\theta_{i,j}+\|P_Kv_{i,j}\|^2/2+c_{i,j}}$ is convex, differentiable and 
\begin{multline}
\label{eq:nabla_phi_ex_1}
\nabla\varphi\colon (\theta,v)\mapsto (e^{\theta_{i,j}+\|P_Kv_{i,j}\|^2/2+c_{i,j}}(1,P_Kv_{i,j})^{\top})_{0\le i,j\le N-1}\\
=\nabla\varphi(\theta,(P_Kv_{i,j})_{0\le i,j\le N-1})
\end{multline}
is $L_{M_0}(\theta^*,v^*)$-Lipschitz continuous in $B((\theta^*,v^*),M_0)$, for every $M_0>0$.
	However, for every $M_0>0$, $\varphi$ is not strongly convex in $B((\theta^*,v^*),M_0)$ (i.e., $\mu_{M_0}(\theta^*,v^*)=0$). Indeed,
	let $\theta=\theta^*$ and let $v\in (\RR_-\times\RR_+\times\RR_-\times\RR_+)^N\setminus\{0\}$. Then, for every $i,j\in\{0,\ldots,N-1\}$, $P_Kv_{i,j}=0$ and \eqref{e:solgomes} yields
	\begin{align}
		&(\nabla\varphi(\theta,v)-\nabla\varphi(\theta^*,v^*))_{i,j}\nonumber\\
		&=e^{\theta_{i,j}+\|P_Kv_{i,j}\|^2/2+c_{i,j}}(1,P_Kv_{i,j})^{\top}-e^{\theta_{i,j}^*+\|P_Kv_{i,j}^*\|^2/2+c_{i,j}}(1,P_Kv_{i,j}^*)^{\top}\nonumber\\
		&=e^{t^*+c_{i,j}}(1,0)^{\top}-e^{t^*+c_{i,j}}(1,0)^{\top}=0
	\end{align}
	and hence
	\begin{equation}
		\scal{\nabla\varphi(\theta,v)-\nabla\varphi(\theta^*,v^*)}{(\theta-\theta^*,v-v^*)}=0.
	\end{equation}
	Then, \eqref{e:form1dex} reduces to Problem~\ref{prob:main} but without any strong convexity.
	In this case, Algorithm~\ref{alg:FBS1d} reduces to
	\begin{equation}
		\label{e:FBS_form2dex}\tag{DFB1}
		(\forall n\in\NN)\quad 
		\begin{array}{l}
			\left\lfloor
			\begin{array}{l}
				\text{for }i,j=0,\ldots,N-1\\
				\left\lfloor
				\begin{array}{l}
					t_{i,j,n} = \overline{\theta}_{i,j,n}-\gamma e^{\overline{\theta}_{i,j,n}+\|P_K\overline{v}_{i,j,n}\|^2/2+c_{i,j}}\\
					\nu_{i,j,n}=\overline{v}_{i,j,n}-\gamma e^{\overline{\theta}_{i,j,n}+\|P_K\overline{v}_{i,j,n}\|^2/2+c_{i,j}}P_K\overline{v}_{i,j,n}\\
				\end{array}
				\right.\\
				(\overline{\theta}_{n+1},\overline{v}_{n+1})=\prox_{\gamma\sigma_{D_{\RR^4}}}((t_{i,j,n},\nu_{i,j,n})_{0\le i,j\le N-1}),
			\end{array}
			\right.
		\end{array}
	\end{equation}
	where $(\overline{\theta}_{0},\overline{v}_{0})\in \M_{h}\times\mathcal{W}_{h}$ and $\gamma >0$. In this context,
	if $\gamma\in\,]0,2/L_{\overline{M}_0}(\theta^*,v^*)[$, with $\overline{M}_0=\|(\overline{\theta}_{0},\overline{v}_{0})-({\theta}^{*},{v}^{*})\|$}, Theorem~\ref{t:conv}\eqref{t:convi} guarantees 
the convergence of the sequence $(\overline{\theta}_n,\overline{v}_n)_{n\in\NN}$ generated by \eqref{e:FBS_form2dex}
to $(\theta^*,v^*)$. Our result does not guarantee linear convergence, but the computation of 
$\prox_{\gamma\sigma_{D_{\RR^4}}}$ is obtained from the projection onto $D_{\RR^4}$, which only needs 
a numerically efficient matrix inversion.

In Figure~\ref{fig:gomes:cr} we compare the numerical behavior of the algorithms in \eqref{e:FBS_form1dex}, \eqref{e:FBS_form2dex}, and the theoretical upper bound obtained in Theorem~\ref{t:conv}\eqref{t:convii}. We set $\widetilde{x}_{n}=(\widetilde{\theta}_{n},\widetilde{v}_{n})$ and $\overline{x}_{n}=(\overline{\theta}_{n},\overline{v}_{n})$ for the iterates computed with~\eqref{e:FBS_form1dex} and~\eqref{e:FBS_form2dex}, respectively. We consider the step-size $\gamma^*$ achieving the optimal convergence rate in Remark~\ref{rem:optrate}\eqref{rem:optratei} and, for \eqref{e:FBS_form2dex}, we set $\gamma=1.99/L_{\overline{M}_0}(\theta^*,v^*)$.
We also set the error tolerance to $7\cdot 10^{-5}$ and $\widetilde{x}_0=(\widetilde{\theta}_0,\widetilde{v}_0)=(\overline{\theta}_0,\overline{v}_0)=\overline{x}_0$. In each case, the initial point $x_0=\widetilde{x}_0,\,\overline{x}_0$ is chosen by perturbing  $x^*=(\theta^*,v^*)$ randomly and such that $\|x_0-x^*\|\in \{0.1,0.5\}$. We observe that the numerical and theoretical linear convergence rates are closer for closer starting points, while in both cases the error of 
\eqref{e:FBS_form2dex} decreases slowly with respect to the number of iterations. However, since 
\eqref{e:FBS_form1dex} involves sub-iterations in order to compute the projection onto $D_K$, it is actually much slower in terms of computational time than \eqref{e:FBS_form2dex} as it can be perceived in Table~\ref{t:1}.

\begin{figure}[ht!]
	\centering
	\begin{subfigure}[b]{0.49\textwidth}
		\includegraphics[width=\textwidth]{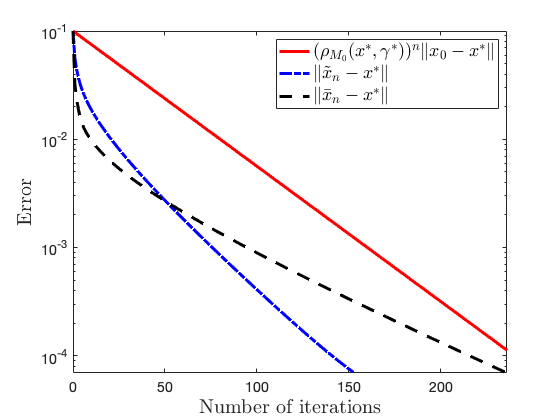}
		\caption{$M_0=\|x_0-x^*\|=0.1$.}
		\label{fig:gomes:cr:01}
	\end{subfigure}
	\begin{subfigure}[b]{0.49\textwidth}
		\includegraphics[width=\textwidth]{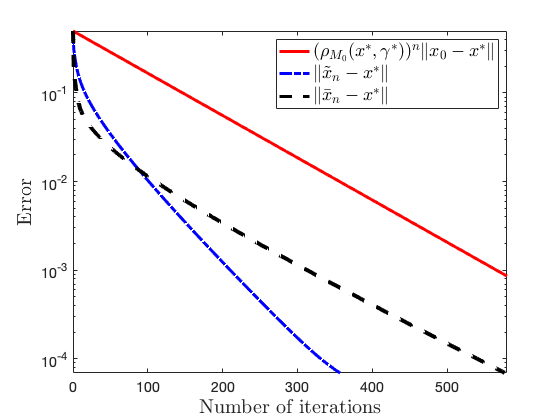}
		\caption{$M_0=\|x_0-x^*\|=0.5$.}
		\label{fig:gomes:cr:05}
	\end{subfigure}
	\caption{Numerical comparison of the theoretical upper bound and the errors of the algorithms in \eqref{e:FBS_form1dex} and \eqref{e:FBS_form2dex} for a tolerance of $7\cdot 10^{-5}$ and two initial conditions.}
	\label{fig:gomes:cr}
\end{figure}

\begin{table}[ht!]\scriptsize
	\centering
	\begin{tabular}{c c c c c c c c c}
		\hline
		& &  &  &   &  &  &  & \\
		&  &  \multicolumn{3}{c}{DFB0} & &\multicolumn{3}{c}{DFB1} \\
		\hhline{~|~|-|-|-|~|-|-|-}
		& &  &  &   &  &  & & \\
		$\|x_0-x^*\|$ & & $\gamma$ & Time (s) & Iterations &   & $\gamma$ & Time (s) &  Iterations \\
		& &  &  &   &  &  & & \\
		$0.1$ & & $0.3748$ & $170.343$ & $153$ & & $0.3748$ & $3.609$ & $236$ \\
		$0.5$ & & $0.2132$ & $566.695$ & $357$ & & $0.2132$ & $8.486$ & $577$\\
		\hline
	\end{tabular}
	\vspace{0.2cm}
	\caption{Performance of \eqref{e:FBS_form1dex} and \eqref{e:FBS_form2dex}, for different initialization when the number of grid nodes (DoF) is $60^2$ and the tolerance is $7\cdot 10^{-5}$.  }
	\label{t:1}
\end{table}

\subsection{A second order mean field game with logarithmic and power couplings}
Let $\nu\in\RPP$, $\alpha\in\left]1,2\right]$, and $\epsilon\in\RP$. We consider the MFGs system
\begin{align}
	-\nu\Delta u+\frac{1}{2}\|\nabla u\|^{2}+\lambda&=
	m^{\alpha-1}+\epsilon\log m+\widetilde{H}(x,y)\quad\text{in }\TT^{2},\nonumber\\
	-\nu\Delta m-\text{div}\big(m\nabla u\big)&=0\quad\text{in }\TT^{2},\nonumber\\
	\int_{\TT^2} u(x) {\rm d} x=0, &\:\:\int_{\TT^2}m(x){\rm d} x=1,  \: m>0,
\end{align}
where, for all $(x,y)\in\TT^{2}$,
$$
\widetilde{H}(x,y)=-\frac{\sin(2\pi x)+\sin(2\pi y)+\cos(4\pi x)}{2}.
$$
We consider the finite difference approximation~\eqref{eq:finite_difference_scheme} of the system above, which corresponds to the optimality condition of Problem~\ref{prob:MFG} with $\funp=|\cdot|^2/2+\iota_{\RP}$ and 
\begin{equation}
	F_{i,j}\colon \rho\mapsto 
	\begin{cases}
		\frac{1}{\alpha}\rho^\alpha+c_{i,j}\rho+\epsilon\, \rho(\log \rho-1),&\text{if}\:\:\rho>0;\\
		0, &\text{if}\:\:\rho=0;\\
		\pinf, &\text{if}\:\:\rho<0,
	\end{cases}    
\end{equation}
where, for every $i,j\in\{0,\ldots,N-1\}$, $c_{i,j}=\widetilde{H}(ih,jh)$.  

	\begin{table}[ht!]\Huge
		\centering
		\resizebox{\textwidth}{!}{
			\begin{tabular}{c c c c c c c c c c c c c c c c}
				\hline
				& &  &  &   &  &  &  & & & & & & & &   \\
				& \multicolumn{3}{c}{CP} & & \multicolumn{3}{c}{DR} & &  \multicolumn{3}{c}{DFB0} & &  \multicolumn{3}{c}{DFB1} \\
\hhline{~|-|-|-|~|-|-|-|~|-|-|-|~|-|-|-}
				&  &  &   &  &  &  &  & & & & & & & &  \\
				$\nu$    & Time (s) & Iterations & $\gamma$ &  & Time (s) & Iterations & $\gamma$ & & Time (s) & Iterations & $\gamma$ & & Time (s) & Iterations & $\gamma$ \\
				& &  &  &   &  &  &  & & & & & & & &    \\
				$0.1$ & $5.285$  & $26$ & $1.05$  &   & $5.2094$ & $26$ & $1.05$ & & $34.4958$ & $23$ & $0.6$  &  & $1.6456$ & $30$  & $0.65$     \\
				$0.5$ & $3.1873$ & $16$   & $1.05$  &   & $3.1924$ & $16$ & $1.05$  & & $17.4105$ & $12$  & $0.55$  & & $0.50957$ & $10$ & $0.55$  \\
				\hline
		\end{tabular}}
		\vspace{0.2cm}
		\caption{Performance comparison among CP, DR, \eqref{e:FBS_form1dex}, and \eqref{e:FBS_form2dex}, for varying values of $\nu$ when $\epsilon = 0$ and the number of grid nodes (DoF) is $60^2$.  }
		\label{tb:rel0}
	\end{table}

	\begin{table}[ht!]\Huge
		\centering
		\resizebox{\textwidth}{!}{
			\begin{tabular}{c c c c c c c c c c c c c c c c}
				\hline
				& &  &  &   &  &  &  & & & & & & & &   \\
				& \multicolumn{3}{c}{CP} & & \multicolumn{3}{c}{DR} & &  \multicolumn{3}{c}{DFB0} & &  \multicolumn{3}{c}{DFB1} \\
\hhline{~|-|-|-|~|-|-|-|~|-|-|-|~|-|-|-}
				&  &  &   &  &  &  &  & & & & & & & &  \\
				$\nu$    & Time (s) & Iterations & $\gamma$ &  & Time (s) & Iterations & $\gamma$ & & Time (s) & Iterations & $\gamma$ & & Time (s) & Iterations & $\gamma$ \\
				& &  &  &   &  &  &  & & & & & & & &    \\
				$0.1$ & $5.6857$ & $26$ & $0.95$  &    & $5.5486$ & $26$ & $0.95$ & & $32.6036$  & $22$ & $0.7$ & & $1.4765$  & $27$ & $0.75$    \\
				$0.5$ & $3.5337$ & $17$   & $0.95$  &   & $3.5753$ & $17$ & $0.95$  & & $13.831$ & $10$  & $0.65$ & & $0.48746$ & $10$ & $0.65$  \\
				\hline
		\end{tabular}}
		\vspace{0.2cm}
		\caption{Performance comparison among CP, DR, \eqref{e:FBS_form1dex}, and \eqref{e:FBS_form2dex}, for varying values of $\nu$ when $\epsilon = 0.1$ and the number of grid nodes (DoF) is $60^2$.  }
		\label{tb:rel01}
	\end{table}
	
	\begin{table}[ht!]\Huge
		\centering
		\resizebox{\textwidth}{!}{
			\begin{tabular}{c c c c c c c c c c c c c c c c}
				\hline
				& &  &  &   &  &  &  & & & & & & & &   \\
				& \multicolumn{3}{c}{CP} & & \multicolumn{3}{c}{DR} & &  \multicolumn{3}{c}{DFB0} & &  \multicolumn{3}{c}{DFB1} \\
\hhline{~|-|-|-|~|-|-|-|~|-|-|-|~|-|-|-}
				&  &  &   &  &  &  &  & & & & & & & &  \\
				$\nu$    & Time (s) & Iterations & $\gamma$ &  & Time (s) & Iterations & $\gamma$ & & Time (s) & Iterations & $\gamma$ & & Time (s) & Iterations & $\gamma$ \\
				& &  &  &   &  &  &  & & & & & & & &    \\
				$0.1$ & $7.9191$ & $25$ & $0.8$  &   & $7.561$ & $25$ & $0.8$ & & $25.3507$ & $17$ & $1$ & & $1.3145$ & $20$ & $1.1$    \\
				$0.5$ & $5.6428$ & $20$  & $0.75$  &   & $5.4536$ & $20$ & $0.75$  & & $15.2802$  & $8$ & $1.05$ & & $0.49585$ & $8$ & $1.1$   \\
				\hline
		\end{tabular}}
		\vspace{0.2cm}
		\caption{Performance comparison among CP, DR, \eqref{e:FBS_form1dex}, and \eqref{e:FBS_form2dex}, for varying values of $\nu$ when $\epsilon = 0.5$ and the number of grid nodes (DoF) is $60^2$.  }
		\label{tb:rel05}
	\end{table}

\begin{figure}[ht!]
	\centering
	\begin{subfigure}[b]{0.49\textwidth}
		\includegraphics[width=\textwidth]{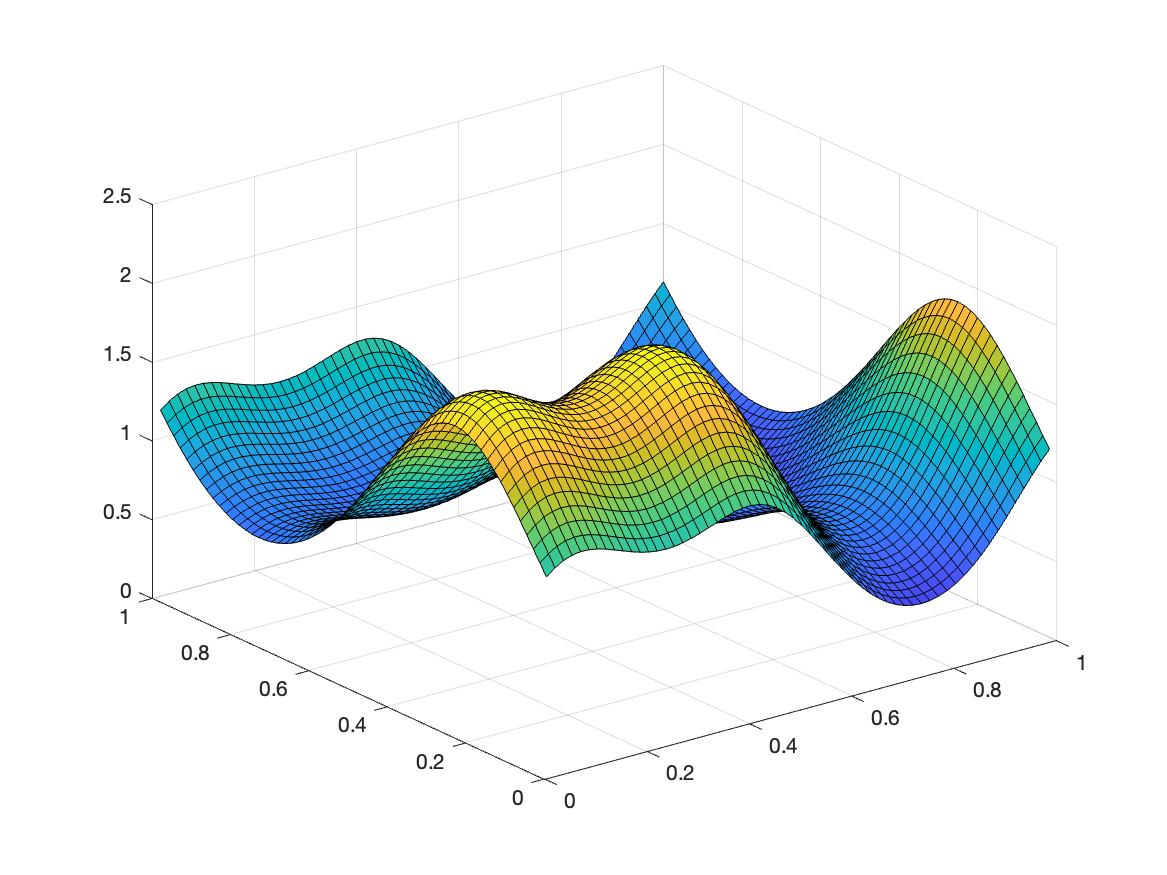}
		\caption{$\epsilon=0,\:\:  \nu=0.1$.}
		\label{fig:2nd:01}
	\end{subfigure}
	\begin{subfigure}[b]{0.49\textwidth}
		\includegraphics[width=\textwidth]{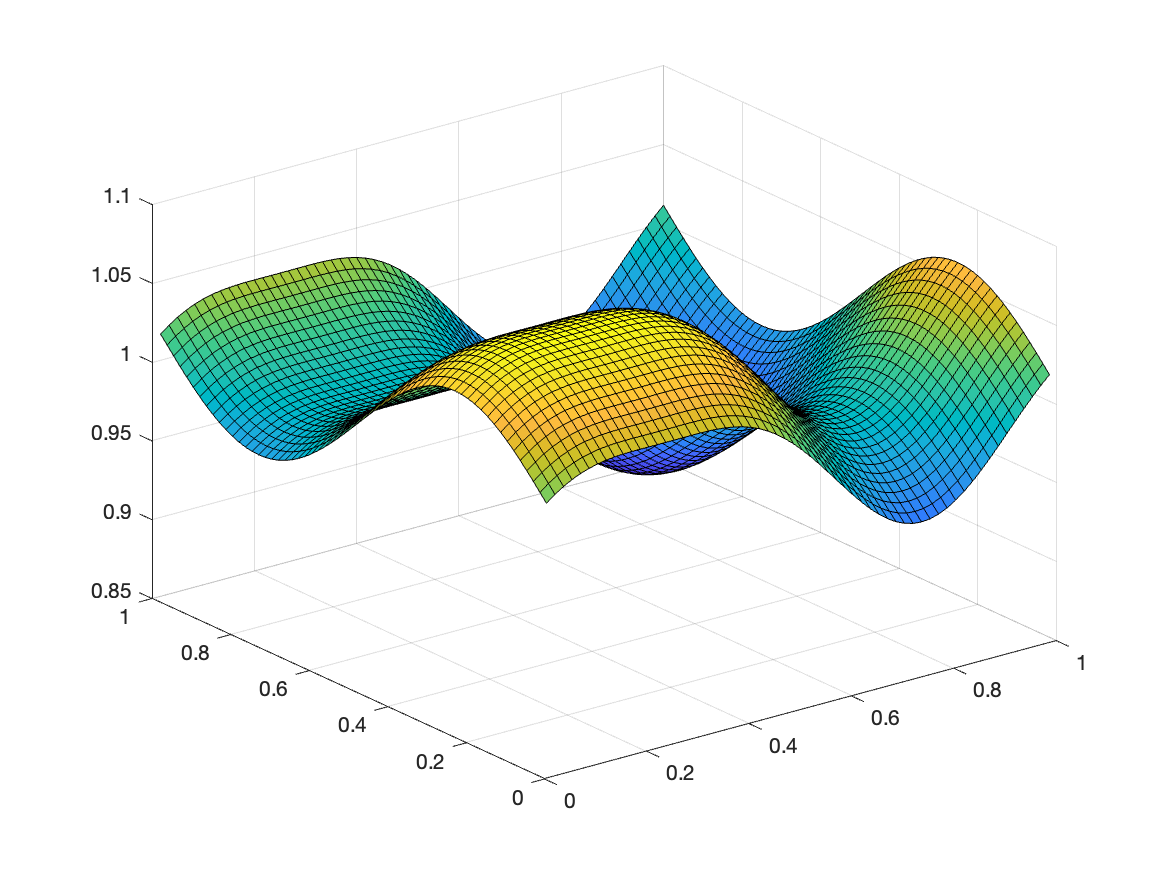}
		\caption{$\epsilon=0,\:\:   \nu=0.5$.}
		\label{fig:2nd:02}
	\end{subfigure}
	\begin{subfigure}[b]{0.49\textwidth}
		\includegraphics[width=\textwidth]{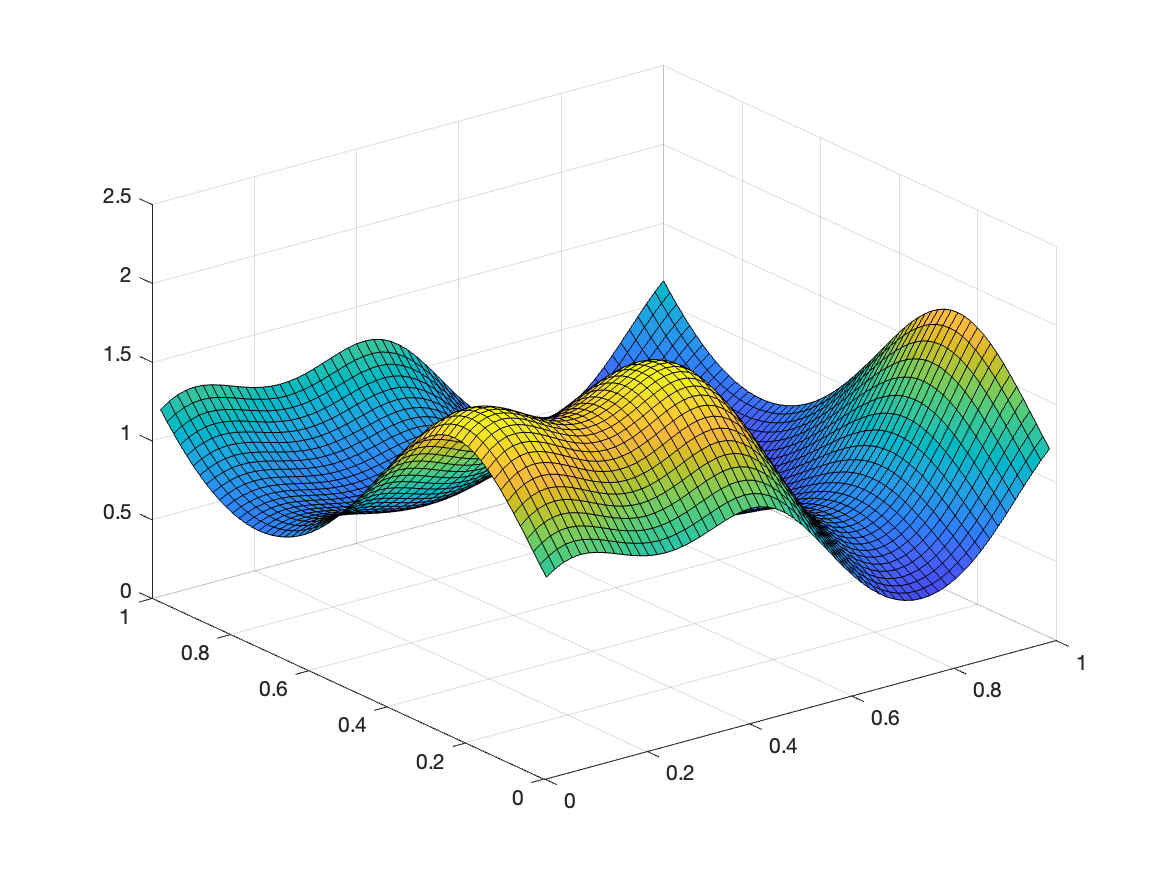}
		\caption{$\epsilon=0.1,\: \:  \nu=0.1$.}
		\label{fig:2nd:03}
	\end{subfigure}
	\begin{subfigure}[b]{0.49\textwidth}
		\includegraphics[width=\textwidth]{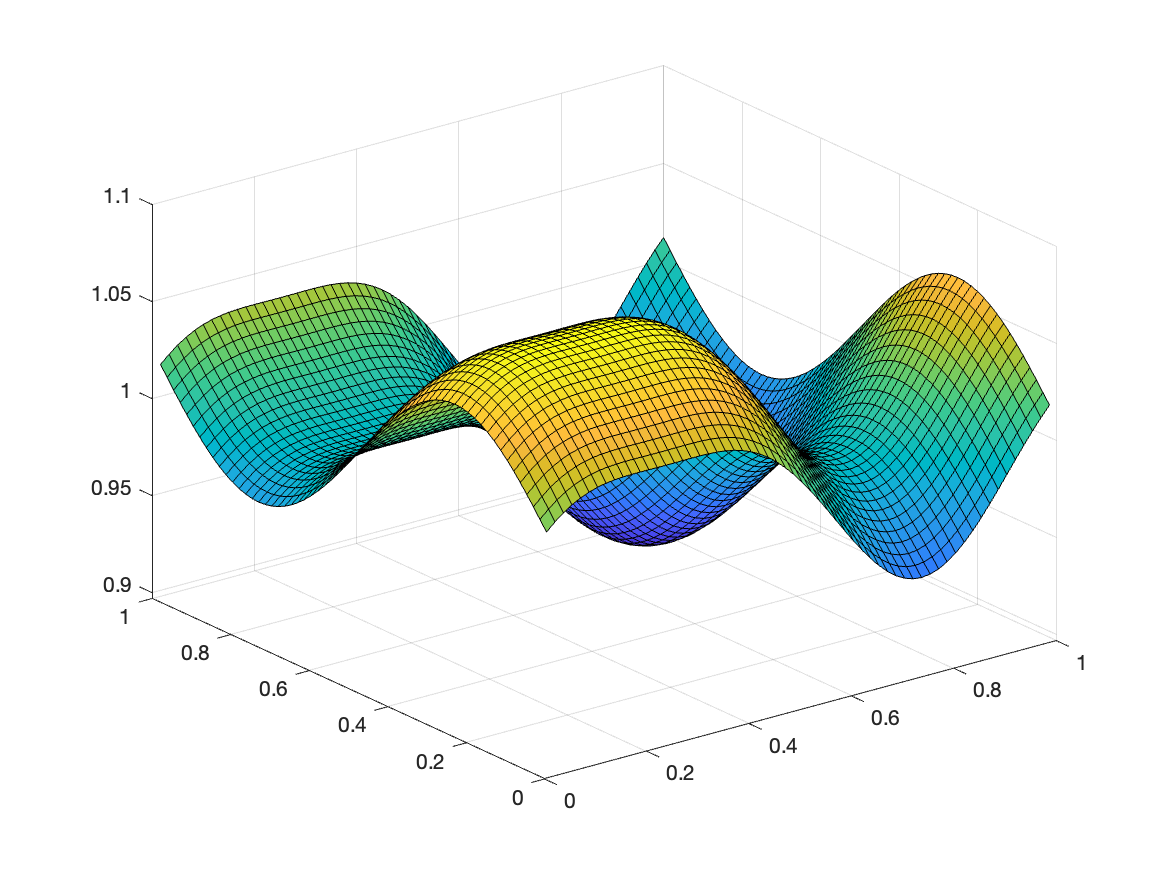}
		\caption{$\epsilon=0.1,\: \:  \nu=0.5$.}
		\label{fig:2nd:04}
	\end{subfigure}
	\begin{subfigure}[b]{0.49\textwidth}
		\includegraphics[width=\textwidth]{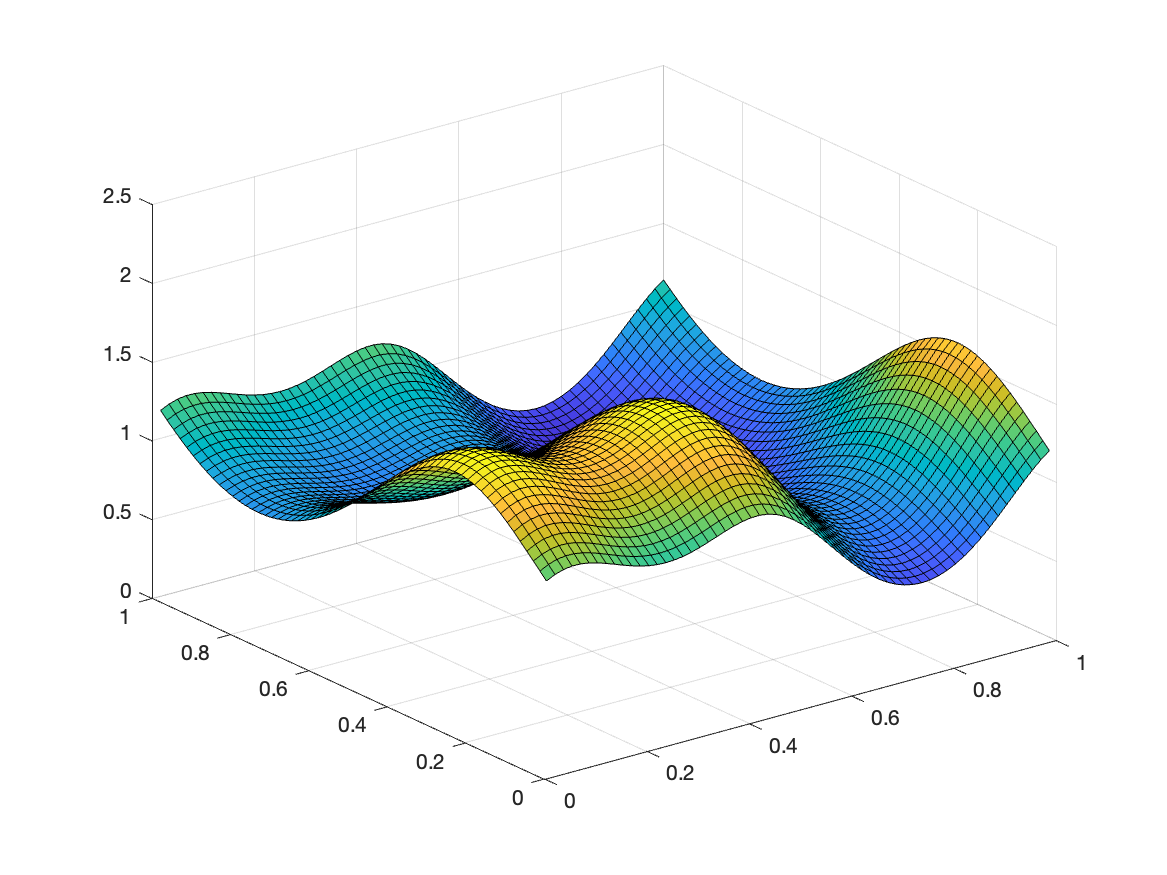}
		\caption{$\epsilon=0.5,\:\:   \nu=0.1$.}
		\label{fig:2nd:05}
	\end{subfigure}
	\begin{subfigure}[b]{0.49\textwidth}
		\includegraphics[width=\textwidth]{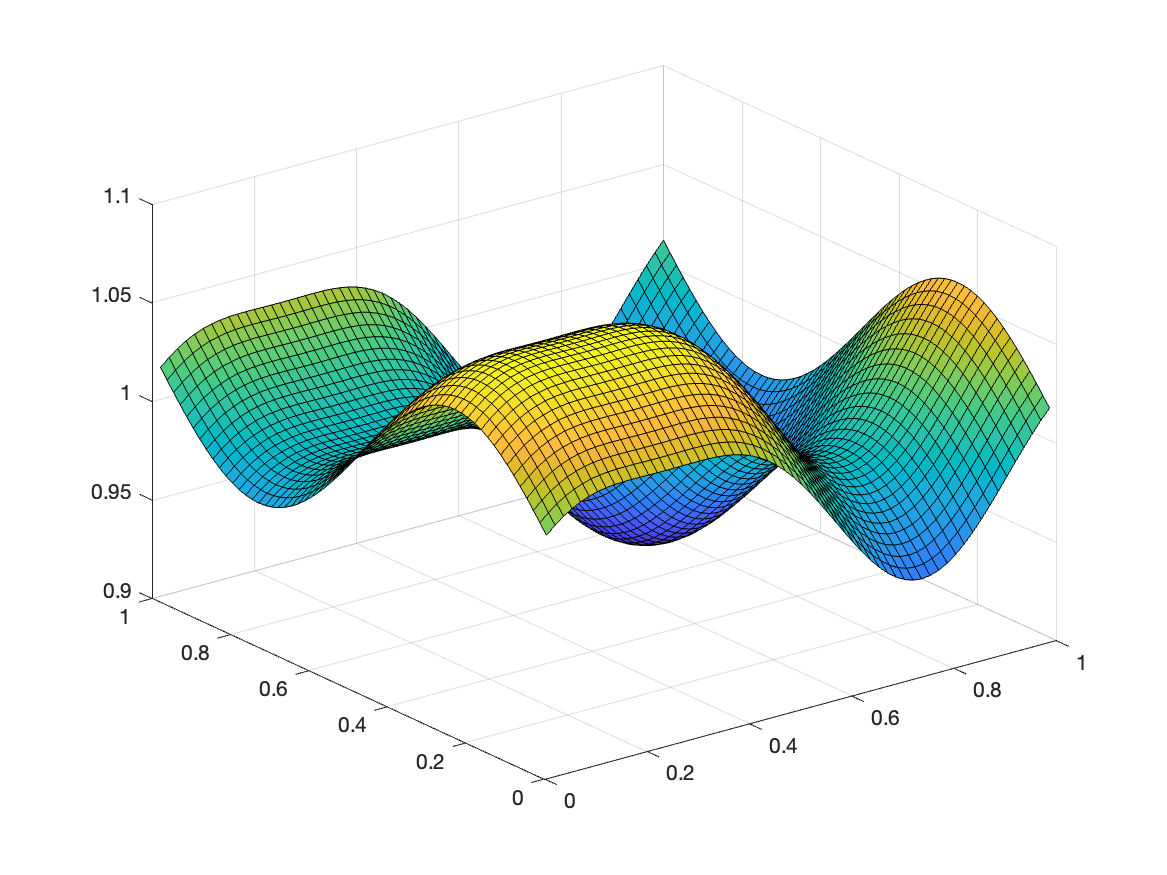}
		\caption{$\epsilon=0.5,\:\:   \nu=0.5$.}
		\label{fig:2nd:06}
	\end{subfigure}
	\vspace{0.2cm}
	\caption{Numerical results of \eqref{e:FBS_form2dex} for different values of $\epsilon$ and $\nu$.}
	\label{fig:gomes:cr_bis}
\end{figure}

Since, for every $i,j\in\{0,\ldots,N-1\}$, $F_{i,j}$ and $\ell$ are supercoercive and strictly convex on their domains, Proposition~\ref{p:0}\eqref{p:0ii} implies that $\phi^*_{i,j}\colon (\theta_{i,j},v_{i,j})\mapsto F_{i,j}^*(\theta_{i,j}+\|v_{i,j}\|^2/2)$ is everywhere defined, differentiable, and
\begin{equation}
	\label{e:gradphist}
	\nabla\phi_{i,j}^*\colon (\theta_{i,j},v_{i,j})\mapsto
	(F_{i,j}^*)'(\theta_{i,j}+\|v_{i,j}\|^2/2)
	\begin{pmatrix}1\\v_{i,j}
	\end{pmatrix}.
\end{equation}
In the case when $\epsilon >0$, we have, for every $i,j\in\{0,\ldots,N-1\}$, $\dom\partial F_{i,j}=\ran (F_{i,j}^*)' =\RPP$  and, therefore, for every $(\rho,y)\in\RR\times\RPP$,
\begin{align}
	& \quad\quad \,\,\  y=(F_{i,j}^*)'(\rho)\nonumber\\
	\quad&\Leftrightarrow\quad 
	\rho=F_{i,j}'(y)=y^{\alpha-1}+c_{i,j}+\epsilon\ln(y)\nonumber\\
	&\Leftrightarrow\quad 
	\Big(\frac{\alpha-1}{\epsilon}\Big)(\rho-c_{i,j})+\ln\Big(\frac{\alpha-1}{\epsilon}\Big)=\Big(\frac{\alpha-1}{\epsilon}\Big)y^{\alpha-1}+\ln\bigg(\Big(\frac{\alpha-1}{\epsilon}\Big)y^{\alpha-1}\bigg)\nonumber\\
	&\Leftrightarrow\quad 
	\Big(\frac{\alpha-1}{\epsilon}\Big)e^{\big(\frac{\alpha-1}{\epsilon}\big)(\rho-c_{i,j})}=\bigg(\Big(\frac{\alpha-1}{\epsilon}\Big)y^{\alpha-1}\bigg)e^{\big(\frac{\alpha-1}{\epsilon}\big)y^{\alpha-1}}\nonumber\\
	&\Leftrightarrow\quad 
	y=\Big(\frac{\epsilon}{\alpha-1}\Big)^{\frac{1}{\alpha-1}}W_0\bigg(\Big(\frac{\alpha-1}{\epsilon}\Big)e^{\big(\frac{\alpha-1}{\epsilon}\big)(\rho-c_{i,j})}\bigg)^{\frac{1}{\alpha-1}},
\end{align}
where $W_0$ stands for the principal branch of the Lambert W-function.
On the other hand, in the case $\epsilon=0$, we have $\dom\partial F_{i,j}=\ran (F_{i,j}^*)' =\RP$.  Altogether, it follows from simple computations that 
\begin{align}
	(F_{i,j}^*)'\colon\rho\mapsto 
	\begin{cases}
		0,&\text{if}\:\:\epsilon=0\:\:\text{and}\:\:\rho<c_{i,j};\\
		(\rho-c_{i,j})^{\frac{1}{\alpha-1}},&\text{if}\:\:\epsilon=0\:\:\text{and}\:\:\rho\ge c_{i,j};\\
		\Big(\frac{\epsilon}{\alpha-1}\Big)^{\frac{1}{\alpha-1}}W_0\bigg(\Big(\frac{\alpha-1}{\epsilon}\Big)e^{\big(\frac{\alpha-1}{\epsilon}\big)(\rho-c_{i,j})}\bigg)^{\frac{1}{\alpha-1}},&\text{if}\:\:\epsilon>0.
	\end{cases}
\end{align}
Since $\alpha\in\left]1,2\right]$, for every $\epsilon\geq 0$ and $i,j\in\{0,\ldots,N-1\}$, the function $(F_{i,j}^*)'$ is Lipschitz continuous on bounded sets. Thus, we derive from \eqref{e:gradphist} that Problem~\ref{prob:MFG} is a particular instance of 
Problem~\ref{prob:main}.

As in Section~\ref{subsec:ex1}, we denote by \eqref{e:FBS_form1dex} and \eqref{e:FBS_form2dex} the Algorithm~\ref{alg:FBS1d} with $(C_1,C_2)=(K,\RR^4)$ and $(C_1,C_2)=(\RR^4,K)$, respectively. In Tables~\ref{tb:rel0}-\ref{tb:rel05}, we compare the computational time and number of iterations that algorithms
\eqref{e:FBS_form1dex}, \eqref{e:FBS_form2dex}, Chambolle-Pock (CP) \cite{12.CP}, and Douglas-Rachford (DR) \cite{lions} take to approximate a solution for the case when $\alpha=1.5$ with a tolerance of $\varepsilon=h^3$ for the relative error $\|(\theta_{n+1}-\theta_n,v_{n+1}-v_n)\|/\|(\theta_n,v_n)\|$. 
For CP, we use critical step-sizes, i.e., $\sigma\gamma=1$, where $\sigma>0$ and $\gamma>0$ are the step-sizes associated with dual and primal variable updates, respectively. This choice is justified by theoretical results ensuring convergence of CP with critical step-sizes \cite{JGO1,Siopt3} and numerical results asserting that the numerical behavior of CP improves as $\sigma\gamma$ approaches $1$ \cite{Siopt3}. Moreover, CP and DR show identical 
numerical behavior, which is justified by the fact that
CP coincides with DR for critical step-sizes \cite{12.CP,Vand}.
We also observe that \eqref{e:FBS_form1dex} needs less number of iterations to achieve a relative error of $\varepsilon=h^3$, however its computational time is much larger because of the subiterations needed to compute the projection onto $D_K$. By considering computational time,~\eqref{e:FBS_form2dex} is the more efficient algorithm, obtaining a time reduction up to 90\% with respect to CP and DR. Moreover, note that the computational time of CP and DR increases as long as the parameter $\epsilon$ multiplying the entropy penalization increases. On the other hand,  \eqref{e:FBS_form2dex} has similar computational time and even a smaller number of iterations to achieve the tolerance as $\epsilon$ increases. We display in Figure~\ref{fig:gomes:cr_bis} the approximation of the equilibrium configurations $m^{h}$ that we obtain for different values of $\epsilon$ and the viscosity parameter $\nu$. Comparing the distributions on the left of Figure~\ref{fig:gomes:cr_bis} with those on the right, we observe that, as expected, for larger viscosity parameters the diffusive behavior of equilibria prevails with regard to the individual preferences of the agents.

\backmatter

\bmhead{Acknowledgments}

The work of Luis M. Brice\~{n}o-Arias was supported by 
Agencia Nacional de Investigaci\'on y Desarrollo 
(ANID-Chile) under grants
FONDECYT 1190871, FONDECYT 1230257, and by Centro de Modelamiento
Matem\'{a}tico (CMM), ACE210010 and FB210005, BASAL funds for centers
of excellence. The work of F. J. Silva was partially supported by l'Agence Nationale de la Recherche (ANR), project  ANR-22-CE40-0010, and by KAUST through the subaward agreement ORA-2021-CRG10-4674.6. X. Yang was partly funded by the Air Force Office Scientific Research under MURI award number FA9550-20-1-0358 (Machine Learning and Physics-Based Modeling and Simulation).

%


\bibliographystyle{spmpsci} 
\bibliography{sn-bibliography}

\begin{thebibliography}{10}
\providecommand{\url}[1]{{#1}}
\providecommand{\urlprefix}{URL }
\expandafter\ifx\csname urlstyle\endcsname\relax
  \providecommand{\doi}[1]{DOI~\discretionary{}{}{}#1}\else
  \providecommand{\doi}{DOI~\discretionary{}{}{}\begingroup
  \urlstyle{rm}\Url}\fi

\bibitem{MR2888257}
Achdou, Y., Camilli, F., Capuzzo-Dolcetta, I.: Mean field games: numerical
  methods for the planning problem.
\newblock SIAM J. Control Optim. \textbf{50}(1), 77--109 (2012).
\newblock \doi{10.1137/100790069}.
\newblock \urlprefix\url{http://dx.doi.org/10.1137/100790069}

\bibitem{Achdou_Camilli_Capuzzo_Dolcetta_13}
Achdou, Y., Camilli, F., Capuzzo-Dolcetta, I.: Mean field games: convergence of
  a finite difference method.
\newblock SIAM J. Numer. Anal. \textbf{51}(5), 2585--2612 (2013).
\newblock \doi{10.1137/120882421}.
\newblock \urlprefix\url{https://doi.org/10.1137/120882421}

\bibitem{Achdou_Capuzzo_Dolcetta_2010}
Achdou, Y., Capuzzo-Dolcetta, I.: Mean field games: numerical methods.
\newblock SIAM J. Numer. Anal. \textbf{48}(3), 1136--1162 (2010).
\newblock \doi{10.1137/090758477}

\bibitem{MR4214773}
Achdou, Y., Cardaliaguet, P., Delarue, F., Porretta, A., Santambrogio, F.: Mean
  field games, \emph{Lecture Notes in Math.}, vol. 2281.
\newblock Springer, Cham; Centro Internazionale Matematico Estivo (C.I.M.E.),
  Florence (2020).
\newblock \doi{10.1007/978-3-030-59837-2}.
\newblock \urlprefix\url{https://doi.org/10.1007/978-3-030-59837-2}

\bibitem{Achdou_Porretta_16}
Achdou, Y., Porretta, A.: Convergence of a finite difference scheme to weak
  solutions of the system of partial differential equations arising in mean
  field games.
\newblock SIAM J. Numer. Anal. \textbf{54}(1), 161--186 (2016).
\newblock \doi{10.1137/15M1015455}.
\newblock \urlprefix\url{https://doi.org/10.1137/15M1015455}

\bibitem{Gomesex}
Almulla, N., Ferreira, R., Gomes, D.: Two numerical approaches to stationary
  mean-field games.
\newblock Dyn. Games Appl. \textbf{7}(4), 657--682 (2017).
\newblock \doi{10.1007/s13235-016-0203-5}.
\newblock
  \urlprefix\url{https://doi-org.usm.idm.oclc.org/10.1007/s13235-016-0203-5}

\bibitem{MR500279}
Baillon, J.B., Haddad, G.: Quelques propri\'{e}t\'{e}s des op\'{e}rateurs
  angle-born\'{e}s et {$n$}-cycliquement monotones.
\newblock Israel J. Math. \textbf{26}(2), 137--150 (1977).
\newblock \doi{10.1007/BF03007664}.
\newblock \urlprefix\url{https://doi.org/10.1007/BF03007664}

\bibitem{Bala63}
Balakrishnan, A.V.: An operator theoretic formulation of a class of control
  problems and a steepest descent method of solution.
\newblock J. SIAM Control Ser. A \textbf{1}, 109--127 (1963) (1963)

\bibitem{19.Livre1}
Bauschke, H.H., Combettes, P.L.: Convex Analysis and Monotone Operator Theory
  in {H}ilbert Spaces, second edn.
\newblock CMS Books in Mathematics/Ouvrages de Math\'{e}matiques de la SMC.
  Springer, Cham (2017).
\newblock \doi{10.1007/978-3-319-48311-5}

\bibitem{Yunier1}
Bello~Cruz, J.Y., Nghia, T.T.A.: On the convergence of the forward-backward
  splitting method with linesearches.
\newblock Optim. Methods Softw. \textbf{31}(6), 1209--1238 (2016).
\newblock \doi{10.1080/10556788.2016.1214959}.
\newblock
  \urlprefix\url{https://doi-org.usm.idm.oclc.org/10.1080/10556788.2016.1214959}

\bibitem{Yunier2}
Bello-Cruz, Y., Li, G., Nghia, T.T.A.: On the linear convergence of
  forward-backward splitting method: {P}art {I}---{C}onvergence analysis.
\newblock J. Optim. Theory Appl. \textbf{188}(2), 378--401 (2021).
\newblock \doi{10.1007/s10957-020-01787-7}.
\newblock
  \urlprefix\url{https://doi-org.usm.idm.oclc.org/10.1007/s10957-020-01787-7}

\bibitem{benamoucarlier15}
Benamou, J.D., Carlier, G.: Augmented {L}agrangian methods for transport
  optimization, mean field games and degenerate elliptic equations.
\newblock J. Optim. Theory Appl. \textbf{167}(1), 1--26 (2015).
\newblock \doi{10.1007/s10957-015-0725-9}

\bibitem{Benamou_et_al_2017}
Benamou, J.D., Carlier, G., Santambrogio, F.: Variational mean field games.
\newblock In: Active particles. {V}ol. 1. {A}dvances in theory, models, and
  applications, Model. Simul. Sci. Eng. Technol., pp. 141--171.
  Birkh\"{a}user/Springer, Cham (2017)

\bibitem{BreSib}
Brezis, H., Sibony, M.: M\'{e}thodes d'approximation et d'it\'{e}ration pour
  les op\'{e}rateurs monotones.
\newblock Arch. Rational Mech. Anal. \textbf{28}, 59--82 (1967/1968).
\newblock \doi{10.1007/BF00281564}.
\newblock \urlprefix\url{https://doi-org.usm.idm.oclc.org/10.1007/BF00281564}

\bibitem{MR4534446}
Briceño-Arias, L., Deride, J., L\'{o}pez-Rivera, S., Silva, F.J.: A
  primal-dual partial inverse algorithm for constrained monotone inclusions:
  applications to stochastic programming and mean field games.
\newblock Appl. Math. Optim. \textbf{87}(2), Paper No. 21, 36 (2023).
\newblock \doi{10.1007/s00245-022-09921-9}.
\newblock \urlprefix\url{https://doi.org/10.1007/s00245-022-09921-9}

\bibitem{MR3968547}
Briceño-Arias, L., Kalise, D., Kobeissi, Z., Lauri\`ere, M.,
  Mateos~Gonz\'{a}lez, A., Silva, F.J.: On the implementation of a primal-dual
  algorithm for second order time-dependent mean field games with local
  couplings.
\newblock In: C{EMRACS} 2017---numerical methods for stochastic models:
  control, uncertainty quantification, mean-field, \emph{ESAIM Proc. Surveys},
  vol.~65, pp. 330--348. EDP Sci., Les Ulis (2019).
\newblock \doi{10.1051/proc/201965330}.
\newblock \urlprefix\url{https://doi.org/10.1051/proc/201965330}

\bibitem{JGO1}
Briceño-Arias, L., Rold\'{a}n, F.: Primal-dual splittings as fixed point
  iterations in the range of linear operators.
\newblock J. Global Optim. \textbf{85}(4), 847--866 (2023).
\newblock \doi{10.1007/s10898-022-01237-w}.
\newblock
  \urlprefix\url{https://doi-org.usm.idm.oclc.org/10.1007/s10898-022-01237-w}

\bibitem{BAKS}
Briceño-Arias, L.M., Kalise, D., Silva, F.J.: Proximal methods for stationary
  mean field games with local couplings.
\newblock SIAM J. Control Optim. \textbf{56}(2), 801--836 (2018).
\newblock \doi{10.1137/16M1095615}

\bibitem{Nelly}
Briceño-Arias, L.M., Pustelnik, N.: Convergence rate comparison of proximal
  algorithms for non-smooth convex optimization with an application to texture
  segmentation.
\newblock IEEE Signal Processing Letters \textbf{29}, 1337--1341 (2022).
\newblock \doi{10.1109/LSP.2022.3179169}

\bibitem{Siopt3}
Briceño-Arias, L.M., Rold\'{a}n, F.: Split-{D}ouglas-{R}achford algorithm for
  composite monotone inclusions and split-{ADMM}.
\newblock SIAM J. Optim. \textbf{31}(4), 2987--3013 (2021).
\newblock \doi{10.1137/21M1395144}.
\newblock \urlprefix\url{https://doi-org.usm.idm.oclc.org/10.1137/21M1395144}

\bibitem{MR3408214}
Cardaliaguet, P.: Weak solutions for first order mean field games with local
  coupling.
\newblock In: Analysis and geometry in control theory and its applications,
  \emph{Springer INdAM Ser.}, vol.~11, pp. 111--158. Springer, Cham (2015).
\newblock \doi{10.1007/978-3-319-06917-3\_5}.
\newblock \urlprefix\url{https://doi.org/10.1007/978-3-319-06917-3_5}

\bibitem{MR2928380}
Cardaliaguet, P., Lasry, J.M., Lions, P.L., Porretta, A.: Long time average of
  mean field games.
\newblock Netw. Heterog. Media \textbf{7}(2), 279--301 (2012).
\newblock \doi{10.3934/nhm.2012.7.279}.
\newblock \urlprefix\url{https://doi.org/10.3934/nhm.2012.7.279}

\bibitem{cardaliaguet_meszaros_santambrogio_2018}
Cardaliaguet, P., M\'esz\'aros, A., Santambrogio, F.: First order mean field
  games with density constraints: Pressure equals price.
\newblock SIAM J. Control Optim. \textbf{54}(5), 2672--2709 (2016).
\newblock \doi{10.1137/15M1029849}

\bibitem{MR3752669}
Carmona, R., Delarue, F.: Probabilistic theory of mean field games with
  applications. {I}, \emph{Probability Theory and Stochastic Modelling},
  vol.~83.
\newblock Springer, Cham (2018).
\newblock Mean field FBSDEs, control, and games

\bibitem{MR3753660}
Carmona, R., Delarue, F.: Probabilistic theory of mean field games with
  applications. {II}, \emph{Probability Theory and Stochastic Modelling},
  vol.~84.
\newblock Springer, Cham (2018).
\newblock Mean field games with common noise and master equations

\bibitem{12.CP}
Chambolle, A., Pock, T.: A first-order primal-dual algorithm for convex
  problems with applications to imaging.
\newblock J. Math. Imaging Vision \textbf{40}(1), 120--145 (2011).
\newblock \doi{10.1007/s10851-010-0251-1}

\bibitem{MR3333058}
Cirant, M.: Multi-population mean field games systems with {N}eumann boundary
  conditions.
\newblock J. Math. Pures Appl. (9) \textbf{103}(5), 1294--1315 (2015).
\newblock \doi{10.1016/j.matpur.2014.10.013}.
\newblock \urlprefix\url{https://doi.org/10.1016/j.matpur.2014.10.013}

\bibitem{Comb18}
Combettes, P.L.: Perspective functions: properties, constructions, and
  examples.
\newblock Set-Valued Var. Anal. \textbf{26}(2), 247--264 (2018).
\newblock \doi{10.1007/s11228-017-0407-x}.
\newblock
  \urlprefix\url{https://doi-org.usm.idm.oclc.org/10.1007/s11228-017-0407-x}

\bibitem{mms05}
Combettes, P.L., Wajs, V.R.: Signal recovery by proximal forward-backward
  splitting.
\newblock Multiscale Model. Simul. \textbf{4}(4), 1168--1200 (2005).
\newblock \doi{10.1137/050626090}.
\newblock \urlprefix\url{https://doi-org.usm.idm.oclc.org/10.1137/050626090}

\bibitem{MR3415027}
Gomes, D.A., Mitake, H.: Existence for stationary mean-field games with
  congestion and quadratic {H}amiltonians.
\newblock NoDEA Nonlinear Differential Equations Appl. \textbf{22}(6),
  1897--1910 (2015).
\newblock \doi{10.1007/s00030-015-0349-7}.
\newblock \urlprefix\url{https://doi.org/10.1007/s00030-015-0349-7}

\bibitem{MR3160525}
Gomes, D.A., Patrizi, S., Voskanyan, V.: On the existence of classical
  solutions for stationary extended mean field games.
\newblock Nonlinear Anal. \textbf{99}, 49--79 (2014).
\newblock \doi{10.1016/j.na.2013.12.016}.
\newblock \urlprefix\url{https://doi.org/10.1016/j.na.2013.12.016}

\bibitem{MR3559742}
Gomes, D.A., Pimentel, E.A., Voskanyan, V.: Regularity theory for mean-field
  game systems.
\newblock SpringerBriefs Math. Springer (2016).
\newblock \doi{10.1007/978-3-319-38934-9}.
\newblock \urlprefix\url{https://doi.org/10.1007/978-3-319-38934-9}

\bibitem{MR2928381}
Gomes, D.A., Pires, G.E., S\'{a}nchez-Morgado, H.: A-priori estimates for
  stationary mean-field games.
\newblock Netw. Heterog. Media \textbf{7}(2), 303--314 (2012).
\newblock \doi{10.3934/nhm.2012.7.303}.
\newblock \urlprefix\url{https://doi.org/10.3934/nhm.2012.7.303}

\bibitem{MR3195844}
Gomes, D.A., Sa\'{u}de, J.: Mean field games models---a brief survey.
\newblock Dyn. Games Appl. \textbf{4}(2), 110--154 (2014).
\newblock \doi{10.1007/s13235-013-0099-2}.
\newblock \urlprefix\url{https://doi.org/10.1007/s13235-013-0099-2}

\bibitem{Huang_et_all}
Huang, M., Caines, P.E., Malham\'{e}, R.P.: Large-population cost-coupled {LQG}
  problems with nonuniform agents: individual-mass behavior and decentralized
  {$\epsilon$}-{N}ash equilibria.
\newblock IEEE Trans. Automat. Control \textbf{52}(9), 1560--1571 (2007).
\newblock \doi{10.1109/TAC.2007.904450}

\bibitem{HMC06}
Huang, M., Malham\'e, R.P., Caines, P.E.: Large population stochastic dynamic
  games: closed-loop {M}c{K}ean-{V}lasov systems and the {N}ash certainty
  equivalence principle.
\newblock Commun. Inf. Syst. \textbf{6}(3), 221--251 (2006)

\bibitem{Lasry_Lions_2006i}
Lasry, J.M., Lions, P.L.: Jeux \`a champ moyen. {I}. {L}e cas stationnaire.
\newblock C. R. Math. Acad. Sci. Paris \textbf{343}(9), 619--625 (2006).
\newblock \doi{10.1016/j.crma.2006.09.019}

\bibitem{MR2271747}
Lasry, J.M., Lions, P.L.: Jeux \`a champ moyen. {II}. {H}orizon fini et
  contr\^{o}le optimal.
\newblock C. R. Math. Acad. Sci. Paris \textbf{343}(10), 679--684 (2006).
\newblock \doi{10.1016/j.crma.2006.09.018}.
\newblock \urlprefix\url{https://doi.org/10.1016/j.crma.2006.09.018}

\bibitem{Lasry_Lions_2007}
Lasry, J.M., Lions, P.L.: Mean field games.
\newblock Jpn. J. Math. \textbf{2}(1), 229--260 (2007).
\newblock \doi{10.1007/s11537-007-0657-8}

\bibitem{Lavigne_Pfeiffer_22}
Lavigne, P., Pfeiffer, L.: Generalized conditional gradient and learning in
  potential mean field games (2022).
\newblock Preprint, arXiv:2209.12772

\bibitem{levitin1966constrained}
Levitin, E.S., Polyak, B.T.: Constrained minimization methods.
\newblock USSR Computational mathematics and mathematical physics
  \textbf{6}(5), 1--50 (1966)

\bibitem{lions}
Lions, P.L., Mercier, B.: Splitting algorithms for the sum of two nonlinear
  operators.
\newblock SIAM J. Numer. Anal. \textbf{16}(6), 964--979 (1979).
\newblock \doi{10.1137/0716071}

\bibitem{Nurbekyan_et_all_2020}
Liu, S., Jacobs, M., Li, W., Nurbekyan, L., Osher, S.J.: Computational
  {M}ethods for {F}irst-{O}rder {N}onlocal {M}ean {F}ield {G}ames with
  {A}pplications.
\newblock SIAM J. Numer. Anal. \textbf{59}(5), 2639--2668 (2021).
\newblock \doi{10.1137/20M1334668}

\bibitem{Mercier}
Mercier, B.: In\'{e}quations variationnelles de la m\'{e}canique,
  \emph{Publications Math\'{e}matiques d'Orsay 80 [Mathematical Publications of
  Orsay 80]}, vol.~1.
\newblock Universit\'{e} de Paris-Sud, D\'{e}partement de Math\'{e}matique,
  Orsay (1980)

\bibitem{MR3420414}
M\'{e}sz\'{a}ros, A.R., Silva, F.J.: A variational approach to second order
  mean field games with density constraints: the stationary case.
\newblock J. Math. Pures Appl. (9) \textbf{104}(6), 1135--1159 (2015).
\newblock \doi{10.1016/j.matpur.2015.07.008}.
\newblock \urlprefix\url{https://doi.org/10.1016/j.matpur.2015.07.008}

\bibitem{meszaros_silva_2018}
M\'{e}sz\'{a}ros, A.R., Silva, F.J.: On the variational formulation of some
  stationary second-order mean field games systems.
\newblock SIAM J. Math. Anal. \textbf{50}(1), 1255--1277 (2018).
\newblock \doi{10.1137/17M1125960}

\bibitem{Vand}
O'Connor, D., Vandenberghe, L.: On the equivalence of the primal-dual hybrid
  gradient method and {D}ouglas-{R}achford splitting.
\newblock Math. Program. \textbf{179}(1-2, Ser. A), 85--108 (2020).
\newblock \doi{10.1007/s10107-018-1321-1}.
\newblock
  \urlprefix\url{https://doi-org.usm.idm.oclc.org/10.1007/s10107-018-1321-1}

\bibitem{Papadakis_Peyre_Oudet_14}
Papadakis, N., Peyr\'{e}, G., Oudet, E.: Optimal transport with proximal
  splitting.
\newblock SIAM J. Imaging Sci. \textbf{7}(1), 212--238 (2014).
\newblock \doi{10.1137/130920058}.
\newblock \urlprefix\url{https://doi.org/10.1137/130920058}

\bibitem{VPA}
P\'{e}rez-Aros, P., Vilches, E.: An enhanced {B}aillon-{H}addad theorem for
  convex functions defined on convex sets.
\newblock Appl. Math. Optim. \textbf{83}(3), 2241--2252 (2021).
\newblock \doi{10.1007/s00245-019-09626-6}.
\newblock \urlprefix\url{https://doi.org/10.1007/s00245-019-09626-6}

\bibitem{MR3623401}
Pimentel, E.A., Voskanyan, V.: Regularity for second-order stationary
  mean-field games.
\newblock Indiana Univ. Math. J. \textbf{66}(1), 1--22 (2017).
\newblock \doi{10.1512/iumj.2017.66.5944}.
\newblock \urlprefix\url{https://doi.org/10.1512/iumj.2017.66.5944}

\bibitem{Rocky_66}
Rockafellar, R.T.: Level sets and continuity of conjugate convex functions.
\newblock Trans. Amer. Math. Soc. \textbf{123}, 46--63 (1966).
\newblock \doi{10.2307/1994612}.
\newblock \urlprefix\url{https://doi.org/10.2307/1994612}

\bibitem{Salzo}
Salzo, S.: The variable metric forward-backward splitting algorithm under mild
  differentiability assumptions.
\newblock SIAM J. Optim. \textbf{27}(4), 2153--2181 (2017).
\newblock \doi{10.1137/16M1073741}.
\newblock \urlprefix\url{https://doi-org.usm.idm.oclc.org/10.1137/16M1073741}

\bibitem{Sibony}
Sibony, M.: M\'{e}thodes it\'{e}ratives pour les \'{e}quations et
  in\'{e}quations aux d\'{e}riv\'{e}es partielles non lin\'{e}aires de type
  monotone.
\newblock Calcolo \textbf{7}, 65--183 (1970).
\newblock \doi{10.1007/BF02575559}.
\newblock \urlprefix\url{https://doi-org.usm.idm.oclc.org/10.1007/BF02575559}

\bibitem{Taylor}
Taylor, A.B., Hendrickx, J.M., Glineur, F.: Exact worst-case convergence rates
  of the proximal gradient method for composite convex minimization.
\newblock J. Optim. Theory Appl. \textbf{178}(2), 455--476 (2018).
\newblock \doi{10.1007/s10957-018-1298-1}.
\newblock
  \urlprefix\url{https://doi-org.usm.idm.oclc.org/10.1007/s10957-018-1298-1}

\end{thebibliography}

\end{document}